\documentclass[12 pt, reqno]{amsart}

\usepackage[dvipsnames]{xcolor}
\usepackage{tikz}
\usepackage{stmaryrd}
\usetikzlibrary{matrix,arrows,decorations.pathmorphing}
\usepackage{graphicx}
\usepackage{mathpazo}
\usepackage{euler}
\usepackage{amssymb}
\usepackage{amsmath,mathtools}
\usepackage{fullpage}
\usepackage{caption}
\usepackage{amsthm}
\usepackage{mathrsfs}
\DeclareMathAlphabet{\mathpzc}{OT1}{pzc}{m}{it}
\usepackage{thmtools}
\usepackage{blkarray}
\usepackage{multirow} 
\usepackage{enumerate}
\usepackage{picinpar} 
\usepackage{tikz-cd}
\usepackage{color}
\usepackage{enumitem} 
\usepackage[toc,page]{appendix}
\usepackage{url} 
\usepackage{multicol}
\usepackage{rotating}
\usepackage[numbers]{natbib}         
\usepackage[colorlinks]{hyperref} 
\hypersetup{colorlinks=true, linkcolor=red, citecolor = OliveGreen, urlcolor = black}
\makeatletter
\newcommand{\customlabel}[2]{%
   \protected@write \@auxout {}{\string \newlabel {#1}{{#2}{\thepage}{#2}{#1}{}} }%
   \hypertarget{#1}{#2}}
\makeatother
\usepackage{setspace}

\usepackage{listings}
\usepackage{comment}
\usepackage{manfnt}
\usepackage{longtable}
\usepackage[OT2,T1]{fontenc}
\DeclareSymbolFont{cyrletters}{OT2}{wncyr}{m}{n}
\DeclareMathSymbol{\Sha}{\mathalpha}{cyrletters}{"58}

\makeatletter
\def\@tocline#1#2#3#4#5#6#7{\relax
  \ifnum #1>\c@tocdepth 
  \else
    \par \addpenalty\@secpenalty\addvspace{#2}%
    \begingroup \hyphenpenalty\@M
    \@ifempty{#4}{%
      \@tempdima\csname r@tocindent\number#1\endcsname\relax
    }{%
      \@tempdima#4\relax
    }%
    \parindent\z@ \leftskip#3\relax \advance\leftskip\@tempdima\relax
    \rightskip\@pnumwidth plus4em \parfillskip-\@pnumwidth
    #5\leavevmode\hskip-\@tempdima
      \ifcase #1
       \or\or \hskip 1em \or \hskip 2em \else \hskip 3em \fi%
      #6\nobreak\relax
    \dotfill\hbox to\@pnumwidth{\@tocpagenum{#7}}\par
    \nobreak
    \endgroup
  \fi}
\makeatother

\numberwithin{equation}{subsection}
\newtheorem{thmx}{Theorem}

\newtheorem{propx}[thmx]{Proposition}

\numberwithin{equation}{subsection}
\newtheorem{theorem}[subsection]{Theorem}

\newtheorem{lemma}[subsection]{Lemma}
\newtheorem{coro}[subsection]{Corollary}
\newtheorem{conjecture}[subsection]{Conjecture}
\newtheorem{prop}[subsection]{Proposition}

\theoremstyle{definition}
\newtheorem{defn}[subsection]{Definition}

\newtheorem{remark}[subsection]{Remark}

\newtheorem{exam}[subsection]{Example}

\newcommand{\mZ}{\mathbf{Z}}

\newcommand{\mQ}{\mathbf{Q}}
\newcommand{\mF}{\mathbf{F}}

\newcommand{\mA}{\mathbf{A}}

\newcommand{\mP}{\mathbf{P}}
\newcommand{\mE}{\mathbf{E}}

\newcommand{\cO}{\mathcal{O}}

\newcommand{\cN}{\mathcal{N}}

\newcommand{\sE}{\mathcal{E}}

\newcommand{\rra}{\rightarrow}
\newcommand{\iso}{\cong}
\newcommand{\aad}{\text{ and }}
\newcommand{\fps}[1]{\llbracket #1 \rrbracket}
\newcommand{\ZZ}[1]{\mathbf{Z} / #1 \mathbf{Z}}
\newcommand{\uZZ}[1]{(\mathbf{Z} / #1 \mathbf{Z})^{\times}}
\newcommand{\brk}[1]{ \left\lbrace #1 \right\rbrace }

\newcommand{\tth}{^{\text{th}}}

\newcommand{\ext}{\hookrightarrow}
\newcommand{\pwr}[1]{\left( #1 \right)}

\newcommand{\lrra}{\longrightarrow}
\newcommand{\unit}{^{\times}}

\newcommand{\cdef}[1]{{\color{black}\textsf{\textit{#1}}}}

\DeclareMathOperator{\Aut}{Aut}

\DeclareMathOperator{\id}{id}

\DeclareMathOperator{\GL}{GL}
\DeclareMathOperator{\PGL}{PGL}

\DeclareMathOperator{\Tr}{Tr}

\DeclareMathOperator{\cyc}{cyc.}

\DeclareMathOperator{\rank}{rk}

\DeclareMathOperator{\tors}{tors}

\DeclareMathOperator{\Gal}{Gal}
\DeclareMathOperator{\spl}{sp}
\DeclareMathOperator{\nsp}{nsp}

\DeclareMathOperator{\Jac}{J}
\DeclareMathOperator{\Spec}{Spec}
\DeclareMathOperator{\pr}{pr}
\DeclareMathOperator{\Prym}{Prym}

\begin{document}

\title{\Large{C}\small{omposite images of Galois for elliptic curves over $\mQ$} \\ \small{\&}\\  \Large{E}\small{ntanglement fields}}
\author{Jackson S. Morrow}
\address{Department of Mathematics and Computer Science, Emory University,
Atlanta, GA 30322}
\email{jmorrow4692@gmail.com}
\date{\today}
\maketitle
\begin{abstract}
Let $E$ be an elliptic curve defined over $\mQ$ without complex multiplication. For each prime $\ell$, there is a representation $\rho_{E,\ell}\colon \Gal(\overline{\mQ}/\mQ) \rra \GL_2(\ZZ{\ell})$ that describes the Galois action on the $\ell$-torsion points of $E$. Building on recent work of Rouse--Zureick-Brown and Zywina, we find models for composite level modular curves whose rational points classify elliptic curves over $\mQ$ with simultaneously non-surjective, composite image of Galois. We also provably determine the rational points on almost all of these curves. Finally, we give an application of our results to the study of entanglement fields.
\end{abstract}

\section{Introduction}
Let $E$ be an elliptic curve over a number field $K$. For any positive integer $n$, we denote the $n$-torsion subgroup of $E(\overline{K})$, where $\overline{K}$ is a fixed algebraic closure of $K$, by $E[n]$. For a prime $\ell$, let 
\begin{equation*}
E[\ell^{\infty}] \coloneqq \varprojlim_{n\geq 1}E[\ell^n]
\end{equation*}
and
\begin{equation*}
E[\tors] \coloneqq \varprojlim_{n\geq 1} E[n].
\end{equation*}
By fixing a $\widehat{\mZ}$-basis for $E[\tors]$, there is an induced $\ZZ{n}$-basis on $E[n]$ for any positive integer $n$. The absolute Galois group $G_{K} \coloneqq \Gal (\overline{K}/K)$ has a natural action on each torsion subgroup, which respects each group structure. In particular, we have the following continuous representations
\begin{align*}
\rho_{E,n}\colon &G_K  \lrra  \Aut(\ZZ{n}) \iso \GL_2(\ZZ{n})  &(\text{mod }n)&,\\
\rho_{E,\ell^{\infty}} \colon & G_K \lrra  \Aut(E[\ell^{\infty}])  \iso \GL_2(\mZ_{\ell}) &(\ell\text{-adic}),&\\
\rho_{E}\colon & G_K \lrra \Aut(E[\tors]) \iso \GL_2(\widehat{\mZ}) &(\text{ad\'elic}),& 
\end{align*}
where the image under $\rho$ is uniquely determined up to conjugacy in its respective general linear group. The \cdef{$n$-division field} $K(E[n])$ is the fixed field of $\overline{K}$ by the kernel of the mod $n$ representation;  moreover, the Galois group of this number field is the image of the mod $n$ representation.

A celebrated theorem of Serre \cite{serre1972} says that for an elliptic curve over $K$ without complex multiplication (non-CM), the ad\'elic representation $\rho_E$ has open image in $\GL_2(\widehat{\mZ})$. Serre's theorem raised many questions concerning the possible images of the ad\'elic representation. The group $\GL_2(\widehat{\mZ})$ is both a product group and a profinite group via the following isomorphisms
$$\prod_{\ell \text{ prime}} \GL_2(\mZ_{\ell}) \iso \GL_2(\widehat{\mZ}) \iso \varprojlim_n \GL_2(\ZZ{n}).$$
Serge Lang \cite{lang1987elliptic} referred to these two characterizations as the \textit{horizontal} and \textit{vertical} natures of $\GL_2(\widehat{\mZ})$, respectively, and this binal nature of $\GL_2(\widehat{\mZ})$ provides two flavors of questions stemming from Serre's work.

Horizontally speaking, for any non-CM elliptic curve over $K$, there exists a smallest integer $r_{E/K} >0$ such that for all $\ell\geq r_{E/K}$ , the $\ell$-adic representation is surjective. Serre asked if $r_{E/K}$ depends only on $K$, and whether $r_{E/\mQ} = 37$. In \cite{zywina2011surjectivity}, Zwyina gave a refined conjecture concerning the surjectivity of the mod $\ell$ image and provided a practical algorithm (implemented in Sage) to compute the finite set of primes $\ell$ for which $\rho_{E,\ell}$ is not surjective; a prime $\ell$ is called \cdef{exceptional} if it belongs to this finite set.

Vertically speaking, one interesting question is to determine when the ad\'elic image is surjective. Serre showed that the ad\'elic image is always contained in some index 2 subgroup of $\GL_2(\widehat{\mZ})$ for $E$ defined over $\mQ$.  Greicius \cite{greicius2010} found necessary and sufficient abstract conditions on a number field $L$ for which $\rho_E$ could be surjective. Building on previous work of Duke \cite{duke1997elliptic} and Jones \cite{jones2010almost}, Zywina \cite{zywina10maximal,zywina2010hilbert} proved that for a number field $L\neq \mQ$ such that $L\cap \mQ^{\cyc} = \mQ$, almost all elliptic curves over $L$ (in the sense of density) have surjective, ad\'elic image.

The vertical variant also leads us to ask for the possible values for the index of the ad\'elic image for a given non-CM elliptic curve. This question is the focus of \cite[Program~B]{mazur1977rational}. In particular, given an open subgroup $H \subset \GL_2(\widehat{\mZ})$, this program strives to classify all elliptic curves $E/K$ such that the image of $\rho_E$ is contained in $H$. The work of this program suggests that there exists a constant $B(K)$ such that for every elliptic curve $E/K$ without complex multiplication, the index of $\rho_{E}(G_K)$ in $\GL_2(\widehat{\mZ})$ is bounded by $B(K)$.

To determine $\rho_E(G_K)$, one begins by computing the $\ell$-adic image $\rho_{E,\ell^{\infty}}$ for each prime $\ell$, which leads to the following inclusions
\begin{equation*}
\rho_{E}(G_K) \ext \prod_{\ell \text{ prime}}\rho_{E,\ell^{\infty}}(G_K) \subseteq \prod_{\ell \text{ prime}}\GL_2(\mZ_{\ell}).
\end{equation*}
The image of $\rho_{E}(G_K)$ under the above inclusion will project onto each $\ell$-adic factor, and so 
a natural first step in Mazur's Program B is to classify the $\ell$-adic image of Galois. 

We briefly recall recent progress in Mazur's Program B. 
Zywina \cite{zywinapossible} has described all known, and conjecturally all, pairs $(E,\ell)$ such that $\rho_{E,\ell}(G_{\mQ})$ is non-surjective. Rouse and Zureick-Brown \cite{rouse2014elliptic} provided a complete list of the 1208 possible 2-adic Galois representations associated to non-CM elliptic curves over $\mQ$. 
Sutherland and Zywina \cite{sutherland5modular} also determined all of the prime power level modular curves $X_G$ for which $X_G(\mQ)$ is infinite.
In all of these works, the authors give rational functions whose values correspond to $j$-invariants of non-CM elliptic curves over $\mQ$ with image of Galois conjugate to a subgroup of $G$ in the appropriate general linear group. The computations of these rational functions occupy the majority of these works. 

In this paper, we investigate the \cdef{composite-$(m_1,m_2)$ image} $(\rho_{E,m_1} \times \rho_{E,m_2})(G_{K})$ for $m_1,m_2$ relatively prime.  
Let $\ell$ be a prime; let $G_{n,\ell}\subset \GL_2(\ZZ{\ell})$ be a proper subgroup which arises as an image of $\rho_{E,\ell}(G_{\mQ})$ and contains $-I$ (these subgroups come from \cite{zywinapossible}, and we produce the list of these subgroups $G_{n,\ell}$ in Appendix \ref{App:AppendixA}; here $\ell$ refers to the level of the group and $n$ is simply an index); and let $H_i\subset \GL_2(\ZZ{2^m})$ be a proper subgroup which arises as an image of $\rho_{E,2^{\infty}}(G_{\mQ})$ and contains $-I$ coming from \cite[\texttt{gl2data.txt}]{RZB}. 
Using the rational functions corresponding to the $j$-maps of the modular curves $X_{H_i}(2^m)$ and $X_{G_{n,\ell}}(\ell)$, construct the fibered product
$$
\begin{tikzcd}
X' \arrow{r} \arrow{d} & X_{G_n}(\ell) \arrow{d}{j(G_{n,\ell})} \\
X_{H_i}(2^m) \arrow[swap]{r}{j(H_i)} & \mP_{\mQ}^1.
\end{tikzcd}
$$
We define the \cdef{composite-$(2^m,\ell)$ level} modular curve $X_{H_i,G_{n,\ell}}(2^m\cdot \ell)$ to be the normalization of the fibered product $X'$. 
The aforementioned $j$-map equations allow us to readily find equations for $X'$, but this curve is usually singular, which necessitates taking a normalization.

The $\mQ$-points on $X_{H_i,G_{n,\ell}}(2^m\cdot \ell)$ correspond to elliptic curves $E$ over $\mQ$ with composite-$(2^m,\ell)$ image conjugate to some subgroup of $H_i \times G_{n,\ell} \subset \GL_2(\ZZ{2^m}) \times \GL_2(\ZZ{\ell}) \iso \GL_2(\ZZ{2^m\cdot \ell})$ via the chinese remainder theorem. Succinctly, these rational points classify elliptic curves over $\mQ$ with simultaneously non-surjective, composite-$(2^m,\ell)$ image of Galois. 

\subsubsection*{Notation} Before we state our main results, we set some notation for specific subgroups of $
\GL_2(\ZZ{\ell})$. Let $C_{\spl}(\ell)$ be the subgroup of diagonal matrices. Let $
\epsilon = -1$ if $\ell \equiv 3 \pmod 4$ and otherwise let $\epsilon \geq 2$ be the 
smallest integer which is not a quadratic residue modulo $\ell$. Let $C_{\nsp}(\ell)$ be the subgroup consisting of matrices of the form $\begin{psmallmatrix} a & b\epsilon \\ b & a \end{psmallmatrix}$ with $(a,b) \in \ZZ{\ell}^2 \setminus \brk{(0,0)}$. Let 
$N_{\spl}(\ell)$ and $N_{\nsp}(\ell)$ be the normalizers of $C_{\spl}(\ell)$ and $C_{\nsp}(\ell)$, respectively, in $\GL_2(\ZZ{\ell})$. We have $[N_{\spl}(\ell) : C_{\spl}(\ell)] = 2$ and the non-identity coset of $C_{\spl}(\ell)$ in $N_{\spl}(\ell)$ is represented by $\begin{psmallmatrix} 0 & 1 \\ 1 & 0 \end{psmallmatrix}$. We have $[N_{\nsp}(\ell) : C_{\nsp}(\ell)] = 2$ and the non-identity coset of $C_{\nsp}(\ell)$ in $N_{\nsp}(\ell)$ is represented by $\begin{psmallmatrix} 1 & 0 \\ 0 & -1 \end{psmallmatrix}$. Let $B(\ell)$ be the subgroup of upper triangular matrices in $\GL_2(\ZZ{\ell})$.

\subsection{Statement of results}
In this paper, we study the possible composite-$(2^m,3)$ for $m=1,2,3,4$ and composite-$(2,\ell)$ for $\ell=5,7,11,13$ images of Galois associated non-CM elliptic curves over $\mQ$.

\begin{thmx}\label{2.2}
Let $E/\mQ$ be a non-CM elliptic curve.
\begin{enumerate}
\item If the composite-$(2,3)$ image of $E$ is simultaneously non-surjective, then the image is conjugate to a subgroup of one of the following subgroups of $\GL_2(\ZZ{6})$:
$$\brk{ \ref{G223} \times \ref{G33},\,  \ref{G222} \times \ref{G23},\,  \ref{G222} \times \ref{G13},\,  \ref{G222} \times \ref{G33},\, \ref{G222} \times \ref{G43},\, \ref{G221} \times \ref{G33}}.$$  \label{2.2.1}
\item If the composite-$(4,3)$ image of $E$ is simultaneously non-surjective, then the image is conjugate to a subgroup of one of the following subgroups of $\GL_2(\ZZ{12})$:
$$\brk{ H_9 \times \ref{G33},\,  H_{10} \times \ref{G33},\,  H_{11}\times\ref{G43},\,H_{12}\times\ref{G43},\,H_{13} \times \ref{G33}}.$$ \label{2.2.2}
\item If the composite-$(8,3)$ image of $E$ is simultaneously non-surjective, then the image is conjugate to a subgroup of one of the following subgroups of $\GL_2(\ZZ{24})$:
$$\brk{ H_{30} \times \ref{G43},\, H_{31} \times \ref{G43},\, H_{39}\times\ref{G43},\,H_{45}\times\ref{G43},\,H_{47} \times \ref{G43},\,  H_{50} \times \ref{G43} }.$$ \label{2.2.3}
\item If the composite-$(16,3)$ image of $E$ is simultaneously non-surjective, then the image is conjugate to a subgroup of one of the following subgroups of $\GL_2(\ZZ{48})$:
$$\brk{ 
\begin{array}{ll}
H_{103} \times \ref{G43},\,  H_{104} \times \ref{G43},\,  H_{105} \times \ref{G43},\, H_{107}\times\ref{G43},\,H_{110}\times\ref{G43},\,H_{112}\times\ref{G43}, \\
H_{113}\times\ref{G43} ,\,H_{114}\times\ref{G43} ,\,H_{150}\times\ref{G43} ,\,H_{153}\times\ref{G43},\,H_{165}\times\ref{G43}   ,\,H_{166}\times\ref{G43}  
\end{array}
}.$$ \label{2.2.4}
\end{enumerate}
\end{thmx}

\begin{propx}\label{C2.2}
Let $E/\mQ$ be a non-CM elliptic curve.
\begin{enumerate}
\item  It occurs infinitely often that the index of $(\rho_{E,2} \times \rho_{E,3})(G_{\mQ})$ in $\GL_2(\ZZ{6})$ is either 4, 8, 9, 12, 18, or 36 .
\item  It occurs infinitely often that the index of $(\rho_{E,4} \times \rho_{E,3})(G_{\mQ})$ in $\GL_2(\ZZ{12})$ divides 18 or 24.
\item  It occurs infinitely often that the index of $(\rho_{E,8} \times \rho_{E,3})(G_{\mQ})$ in $\GL_2(\ZZ{24})$ divides 36.
\item  It occurs infinitely often that the index of $(\rho_{E,16} \times \rho_{E,3})(G_{\mQ})$ in $\GL_2(\ZZ{48})$ divides 72.
\end{enumerate}
\end{propx}

By restricting our attention to non-CM elliptic curves $E$ with a specified mod 2 image of Galois, we can prove additional results on the composite-$(2,\ell)$ image for $\ell=5,7,11,13$. 

\begin{thmx}\label{2.1}
Let $E/\mQ$ be a non-CM elliptic curve. Suppose that $\rho_{E,2}(G_{\mQ})$ conjugate to $\ref{G223}$ i.e., the discriminant of $E$ is a square. 
\begin{enumerate}
\item If the composite-$(2,5)$ image of $E$ is simultaneously non-surjective, then the image is conjugate to a subgroup of $\ref{G223}\times \ref{G59}$ in $\GL_2(\ZZ{10})$. \label{2.1.1}
\item If the composite-$(2,7)$ image of $E$ is simultaneously non-surjective, then the image is conjugate to a subgroup of $\ref{G223}\times \ref{G77}$ in $\GL_2(\ZZ{14})$. \label{2.1.2}
\item If the composite-$(2,11)$ image of $E$ is simultaneously non-surjective, then the image is conjugate to a subgroup of $\ref{G223}\times \ref{G113}$ in $\GL_2(\ZZ{22})$. \label{2.1.3}
\item If the composite-$(2,13)$ image of $E$ is simultaneously non-surjective, then the image is conjugate to a subgroup of $\ref{G223}\times \ref{G137}$ in $\GL_2(\ZZ{26})$. \label{2.1.4}
\end{enumerate}
\end{thmx}

\begin{propx}\label{C2.1}
Let $E/\mQ$ be a non-CM elliptic curve. Suppose that $\rho_{E,2}(G_{\mQ})$ conjugate to $\ref{G223}$ i.e., the discriminant of $E$ is a square. 
\begin{enumerate}
\item  It occurs infinitely often that the index of $(\rho_{E,2} \times \rho_{E,5})(G_{\mQ})$ in $\GL_2(\ZZ{10})$ divides 10.
\item  It occurs infinitely often that the index of $(\rho_{E,2} \times \rho_{E,7})(G_{\mQ})$ in $\GL_2(\ZZ{14})$ divides 16.
\item  It occurs finitely often that the index of $(\rho_{E,2} \times \rho_{E,11})(G_{\mQ})$ in $\GL_2(\ZZ{22})$ divides 110.
\item  It occurs finitely often that the index of $(\rho_{E,2} \times \rho_{E,13})(G_{\mQ})$ in $\GL_2(\ZZ{26})$ divides 182.
\end{enumerate}
\end{propx}

\subsection{Sketch of proof}
The first step in the proofs of Theorems \ref{2.2} and \ref{2.1} is to find models for the composite level modular curves corresponding to the subgroups coming from Rouse--Zureick-Brown \cite{rouse2014elliptic} and Zywina \cite{zywinapossible}. 
Once we have the models for these modular curves, we determine their $\mQ$-points.
The analysis of rational points on this collection of modular curves involves a variety of techniques, which we discuss in Sections \ref{S5} and execute in Sections \ref{sec:gen2}, \ref{sec:gen3}, and \ref{sec:HighGen}. 
The \texttt{Magma} code verifying claims made in these sections can be found at \cite{JSM2017} as well as diagrams summarizing the results of Theorem~\ref{2.2}.

\subsection{Organization of paper} In Section \ref{SBack}, we give a synopsis of the necessary background on modular curves of prime power level. 
In Section \ref{S4}, we construct models for our composite level modular curves.
In Section \ref{S5}, we explain the techniques used to determine these rational points. The subsequent Sections \ref{sec:gen2}, \ref{sec:gen3}, and \ref{sec:HighGen} provide further details of this analysis for curves of increasing genera. 
We conclude in Section \ref{AppEnt} by applying our results to the study of entanglement fields.
In Appendix \ref{App:AppendixA}, we recall relevant background and introduce notation from \cite{zywinapossible}, which we use throughout.

\subsection{Acknowledgments}
This work clearly owes a debt to Rouse--Zureick-Brown \cite{rouse2014elliptic} and Zywina \cite{zywinapossible}.
The author would like to graciously thank his advisor, David Zureick-Brown, for suggesting the problem and for the multitude of helpful conversations on the topic. 
The author extends his thanks to Jeremy Rouse for useful discussions, to Nils Bruin for supplying the proof in Section~\ref{Chab}, and to Maarten Derickx and Lea Beneish for constructive comments on an earlier draft. 
The computations in this paper were performed using the \texttt{Magma} computer algebra system \cite{MR1484478}. The author would also like to thank the referee for their thoughtful comments.

\section{Background}\label{SBack} 
For a subgroup $G \subset \GL_2(\ZZ{\ell})$ with $\det (G) = \ZZ{\ell}\unit$ and $-I \in G$, we can associate to it a modular curve $X_G$, which is a smooth, projective, and geometrically irreducible curve over $\mQ$. It comes with a natural morphism 
\begin{equation*}
\pi_G\colon X_G \lrra \Spec \mQ[j] \cup \brk{\infty} \eqqcolon \mP^1_{\mQ},
\end{equation*}
such that for an elliptic curve $E/\mQ$ with $j_E \notin \brk{0,1728}$, the group $\rho_{E,\ell}(G_{\mQ})$ is conjugate to a subgroup of $G$ if and only if the $j_E = \pi_G(P)$ for some rational point $P \in X_G(\mQ)$. The modular curves $X_G$ of genus 0 with $X_G(\mQ) \neq \emptyset$ are isomorphic to the projective line, and for each such curve, function field is of form $\mQ(h)$ for some modular function $h$ of level $\ell$. Giving the morphism $\pi_G$ is then equivalent to expressing the modular j-invariant in the form $J(h)$.

We now describe a set of necessary conditions on the possible non-surjective images of $\rho_{E,n}(G_{\mQ})$, where $n \geq 2$.

\begin{defn}\label{3.1}
A subgroup $G$ of $\GL_2(\ZZ{n})$ is \cdef{applicable} if it satisfies the following conditions:
\begin{itemize}
\item $G \neq \GL_2(\ZZ{n})$,
\item $-I \in G$ and $\det (G) = \uZZ{n}$,
\item $G$ contains an element with trace 0 and determinant -1 that fixes a point in $(\ZZ{n})^2$ of order $n$.
\end{itemize}
\end{defn}
\begin{prop}[\protect{\cite[Proposition~2.2]{zywinapossible}}] \label{3.2}
Let $E$ be an elliptic curve over $\mQ$ for which $\rho_{E,n}(G_{\mQ})$ is not surjective. Then $\pm \rho_{E,n}(G_{\mQ})$ is an applicable subgroup of $\GL_2(\ZZ{n})$.
\end{prop}

Proposition \ref{3.2} gives necessary conditions for when a proper subgroup of $\GL_2(\ZZ{n})$ can occur as the image of Galois, and hence reduces a part of the problem to a group theoretic computation. From here, Zywina constructs the modular curves corresponding to these subgroups and classifies the rational points on them. This result gives a conjecturally complete description of the \textit{horizontal} flavored question concerning the mod $\ell$ representations. We recall the applicable subgroups $G_{n,\ell}$ of prime level $\ell$ as well as the $j$-map for their associated modular curve $X_{G_{n}}(\ell)$ in Appendix~\ref{App:AppendixA}.
Unless otherwise stated, any subgroup $G_{n,\ell}$ of $\GL_2(\ZZ{\ell})$ will be applicable and come from this list.

In \cite{rouse2014elliptic}, Rouse and Zureick-Brown consider the \textit{vertical} flavored question through their study of the $2$-adic images. The authors determine the possible 2-adic images of Galois by finding all the rational points on the ``tower" of 2-power level modular curves. For a subgroup $H$ of $\GL_2(\widehat{\mZ})$ and an integer $n$ such that $H$ contains the kernel of the reduction map $\GL_2(\widehat{\mZ}) \rra \GL_2(\ZZ{n})$, the authors define $X_H$ to be the quotient of the modular curve $X(n)$ by the image $H(n)$ of $H$ in $\GL_2(\ZZ{n})$. This quotient roughly classifies elliptic curves whose ad\'elic image of Galois is contained in $H$. Furthermore, the authors describe a necessary condition on the $\ell$-adic image.
\begin{defn}\label{3.3}
A subgroup $H \subset \GL_2(\mZ_{\ell})$ is \cdef{arithmetically maximal} if
\begin{enumerate}
\item[$\bullet$]  $\det \colon H \rra \mZ_{\ell}\unit$ is surjective,
\item[$\bullet$]  there is an $M \in H$ with determinant $-1$ and trace zero, and 
\item[$\bullet$]  there is no subgroup $K$ with $H \subseteq K$ so that $X_{K}$ has genus $\geq 2$. 
\end{enumerate}
\end{defn}

Rouse and Zureick-Brown give an equivalent statement to that in Proposition \ref{3.2}. In particular if $E/\mQ$ is an elliptic curve and $H = \rho_{E,2^{\infty}}(G_{\mQ})$, then $H$ is contained in an arithmetically maximal subgroup. The authors determine that there exist 727 arithmetically maximal subgroups of $\GL_2(\mZ_2)$ and give a beautifully detailed diagram of these subgroups (see \cite[Figure~1]{rouse2014elliptic}). 
As above, let $H_i$ denote the $i\tth$ subgroup in their list (as given in \cite[\texttt{gl2data.txt}]{RZB}) and $j(H_i)$ its respective $j$-map; the level of $H_i$ will be clear from context.

\section{Composite level modular curves}\label{S4}
In this section, we discuss models for our composite level modular curves. Recall that the composite-$(2^m,\ell)$ level modular curve is the normalization of the fibered product $X_{G_n}(\ell) \times_{\mP^1_{\mQ}} X_{H_i}(2^m)$ where the maps to $\mP^1_{\mQ}$ are the $j$-maps $j(G_{n,\ell})$ and $j(H_i)$ of $X_{G_n}(\ell)$ and $X_{H_i}(2^m)$, respectively.

\subsection{Models for Theorem \ref{2.2}}\label{M2.2}
In the proof of Theorem \ref{2.2}, we build the ``tower" of $(2^n\cdot 3)$-power level modular curves. First, we compute the rational points on the level $6$ modular curves, which acts as the foundation of our tower. If the subgroup $H\times G \subset \GL_2(\ZZ{2}) \times \GL_2(\ZZ{3}) \iso \GL_2(\ZZ{6})$ occurs as a composite image of Galois, then we find the subgroups of level $4$ from \cite[\texttt{gl2data.txt}]{RZB} that cover $H$ (e.g.,~that contain $H$ in the kernel of reduction).  We find such level $4$ subgroups for all 6 possible composite-$(2,3)$ images and proceed by computing the rational points on the composite-$(4,3)$ level modular curves. We repeat this procedure for each tier of our tower ending with level $16$.

For $n=1$, we sometimes find hyperelliptic models. For $n=2,3,4$, we often find models for the composite-$(2^n,3)$ level modular curves as superelliptic curves defined by the affine equation $y^3 = f(x^2)$.

\subsection{Models for Theorem \ref{2.1}}\label{M2.1}
The discriminant condition allows us to construct hyperelliptic models for the composite-$(2,\ell)$ level modular curves in Theorem \ref{2.1}.
Indeed, an elliptic curve $E/\mQ$ with such a discriminant has 2-division field $\mQ(E[2])$ isomorphic to $\mQ(\alpha)$ where $\alpha$ is a root of the defining cubic equation $f(x)$ of $E$, which is equivalent to $j_E$ being of the form $s^2 + 1728$ for some $s\in \mQ$.
For applicable subgroups $G_{n,\ell} \subset \GL_2(\ZZ{\ell})$, except for the level 11 subgroup $\ref{G113}$, the composite-$(2,\ell)$ level modular curve has the form $$ X_{\ref{G223},G_{n,\ell}}(2\cdot \ell)\colon s^2 + 1728 = f(t)/g(t),$$ where $f,g \in \mQ[t].$ Through some simple manipulation, we rewrite our modular curve as  $g(t)^2s^2 = f(t)g(t) - 1728g(t)^2 = h(t)^2w(t)$ for some $h,w \in \mQ[t]$. Then we consider the birational map 
\begin{eqnarray*}
\varphi \colon  X_{\ref{G223},G_{n,\ell}}(2\cdot\ell) & \lrra & X\\
(s,t) &\longmapsto  & (g(t)s/h(t) , t).
\end{eqnarray*}
Hence we have reduced our problem to finding the rational points on the hyperelliptic curve
\begin{equation*}
X \colon y^2 = w(t).
\end{equation*}

\begin{remark}
In the proofs of Theorems \ref{2.2} and \ref{2.1}, we first consider maximal applicable subgroups. If $H,H' \subseteq \GL_2(\ZZ{\ell})$ are both applicable such that $H$ is maximal and $H' \subset H$, then we have a map between the composite level modular curves $ X_{G,H'} \rra X_{G,H}$. Hence, the points on $X_{G,H'}$ must map to points on $X_{G,H}$. In particular, if $X_{G,H}(\mQ)$ is finite, then so is $X_{G,H'}(\mQ)$. 
\end{remark}

\section{Analysis of Rational Points --- Theory}\label{S5}
The composite level curves whose models we computed have genera ranging from zero to seven.

\subsection{Low genus curves}
For the genus 0 curves, we determine whether the curve has a rational point, and if so we compute an explicit isomorphism with $\mP_{\mQ}^1$. For the genus 1 curves, we determine whether the curves have a non-singular rational point, and if so, we compute a model for the resulting elliptic curve and determine its rank and torsion subgroup. This is straightforward: most of the  covering maps have degree 2, so we end up with a model of the form $y^2 = p(t)$, where $p(t)$ is a polynomial, and the desired technique is implemented in \texttt{Magma}. The remaining cases are handled via other techniques. \\ 

For the higher genera, our toolkit to analyze rational points consists of:
\begin{enumerate}
\item local methods,
\item the Chabauty--Coleman method,
\item quotients,
\item \'etale descent,
\item the Mordell--Weil sieve,
\item Prym varieties.
\end{enumerate}
Below, we describe some of the theory behind these techniques and the subsequent sections provide a case by case analysis of the rational points on our composite level modular curves.

\begin{table}[h!]
\begin{tabular}{|c||c|c|c|c|c|c|c|}
\hline
Type& $(2,3)$ & $(4,3)$ & $(8,3)$ & $(16,3)$ & $(2,5)$ & $(2,7)$ & $(2,11)$\\
\hline\hline
\vspace*{-.2em}$\mP^1$ & 6  & 5 & 4 & 4 & 1 & 1 & \\
Elliptic curve with rank 0 & 2  & 14 & 12 & & &  & \\
Elliptic curve with rank $>0$ &   & & 2&  & &  &  \\
Genus 2 &   & & 8&  & 1 & 3 & \\
Genus 3 and hyperelliptic &   & & & 8& & 2 & \\
Genus 3 and non-hyperelliptic &   & & & 5 & &  & \\
Genus 4 and non-hyperelliptic &   & & & 6& &  & \\
Genus 6 and hyperelliptic &   & & & 2 & &  & \\
Genus 7 and non-hyperelliptic &   & & & & &  & 1\\
\hline
\end{tabular}
\caption{Data of isomorphism classes for composite level modular curves}
\end{table}

\begin{remark}[Facts about rational points on $X_{G,H}$]\label{5.1}
Every rational point on a curve $X_{G,H}$ of genus one that has rank zero is a cusp or a CM point. Also, all the rational points curves of higher genera are either cusps or CM points, and hence there are no \textit{sporadic} points.
\end{remark}

\subsection{The Chabauty--Coleman method} \label{Chabauty}
Let $X/\mQ$ be a smooth, projective, and geometrically integral curve. 
In 1941, Chabauty \cite{chabauty1941points} proved the finiteness of $X(\mQ)$ under the condition that the Jacobian $J$ of $C$ has rank $r\coloneqq \rank_{\mZ}J(\mQ)$ less than the genus $g$ of $X$. 
Chabauty's idea was to consider $X(\mQ)$ inside the more tractable space $X(\mQ_p)\cap \overline{J(\mQ)}$ where $\overline{J(\mQ)}$ is the $p$-adic closure of $J(\mQ)$ inside of $J(\mQ_p)$. To deduce finiteness of this intersection, Chabauty constructed locally analytic functions, which are $p$-adic integrals in modern parlance, vanishing on $X(\mQ)$ and deduced his result utilizing the fact that an analytic function cannot take a value infinitely often. 

Using techniques from $p$-adic analysis, namely Newton polygons, Coleman \cite{coleman1985effective} controlled the zeros of these $p$-adic integrals to give an explicit upper bound on the number of $\mQ$-points of a curve over $\mQ$ when the rank $r\leq g-1$ and $p$ is a prime of good reduction. 
The practical output is that if $r\leq g-1$ , then $p$-adic integration produces an explicit 1-variable power series $f \in \mZ_p \llbracket t\rrbracket$ whose set of $\mZ_p$-solutions contains all of the rational points. 
This is all implemented in \texttt{Magma} for genus 2 curves over $\mQ$, and in Section~\ref{subsec:rank1}, we discuss the documentation.

In Section~\ref{Chab}, we perform an explicit Chabauty computation for a non-hyperelliptic genus 3 curves, so we briefly recall results from $p$-adic integration; 
we refer the reader to \cite{mccallum2007method, katz2016diophantine} for further details. 
Let $C_{\mQ_p}$ denote the base change of $C$ to $\mQ_p$ for $p$ a prime of good reduction.
Given a point $P\in X_{\mF_p}(\mF_p)$, the inverse image of $P$ under the surjective reduction map
$$
\begin{tikzcd}
\rho\colon C(\mQ_p)\arrow[two heads]{r} & C_{\mF_p}(\mF_p)
\end{tikzcd}
$$
is isomorphic to a $p$-adic disk $D_P$; this isomorhpism is induced by the uniformizer $t$ at any point $Q\in D_P$.
Since the $p$-adic disk has trivial de Rham cohomology, any $\omega\in H^0(C_{\mQ_p},\Omega^1)$ can be expressed as a power series on $D_P$:
$$\omega |_{D_P} = \sum_{i=0}^{\infty} a_it^i dt \in \mZ_p\fps{t}dt .$$
Now for $Q_1,Q_2\in D_P$, the $p$-adic integral is defined by formal anti-differentiation as
$$\int_{Q_1}^{Q_2}\omega \coloneqq \int_{t(Q_1)}^{t(Q_2)} \sum_{i=0}^{\infty} a_it^i dt = \left.\pwr{\sum_{i=0}^{\infty} \frac{a_i}{i+1} t^{i+1}}\right|_{t(Q_1)}^{t(Q_2)}.$$

To summarize, the Chabauty--Coleman method states that if $r\leq g-1$ and $p$ is a prime of good reduction, then there exists a $g-r$ dimensional space of differentials $\Lambda_C \subset H^0(C_{\mQ_p},\Omega^1)$ such that the $p$-adic integrals $\int \omega$ vanish on $\mQ$-points of $C$. 
Using results on Newton polygons, we can effectively bound these zeros inside each residue disk.

\subsection{\'Etale descent}\label{5.3}
\'Etale descent is a ``going up" style technique, first studied in \cite{coombes1989heterogeneous, wetherell1997bounding} and developed as a full theory in \cite{skorobogatov2001torsors}. It is now a standard technique for resolving the rational points on curves (cf.~\cite{ flynn2001covering, bruin2003chabauty}). 

Let $\pi\colon X \rra Y$ be a degree $n$ \'etale cover defined over a number field $K$ such that $Y$ is the quotient of some free action of a group $G$ on $X$. By Riemann--Hurwitz, the genus of $X$ is $n g(Y) - (n-1)$. Then there exists a finite collection $\pi_1\colon X_1 \rra Y,\dots , \pi_n\colon X_n \rra Y$ of twist of $X\rra Y$ such that
\begin{equation*}
\bigcup_{i=1}^n \, \pi_i(X_i(K)) = Y(K).
\end{equation*}
We shall use this procedure on in the case of \'etale double covers. In this case, $G = \ZZ{2}$, and since the twists are consequently quadratic, we will instead denote the twist of a double cover $X \rra Y$ by $X_d \rra Y$, where $d \in K\unit/(K\unit)^2$. The above discussion gives that, for any point of $Y(K)$, there will exist $d \in \cO_{K,S}\unit/(\cO_{K,S}\unit)^2$ such that $P$ lifts to a point of $X_{d}(K)$ where $S$ is the union of the sets of primes of bad reduction of $X$ and $Y$ and of the primes of $\cO_K$ lying over 2.

\subsection{The Mordell--Weil sieve}\label{MWsieve}
In many situations, we encounter a curve $C$ with only one known, non-singular $\mQ$-point $\infty$, and we wish to prove that $C(\mQ) = \brk{\infty}$. We can define an Abel--Jacobi map based at $\infty$, which allows us to consider the commuting diagram 
$$
\begin{tikzcd}
C(\mQ) \arrow[right hook->]{r}{\iota} \arrow{d}{\beta} & J(\mQ) \arrow{d}{\alpha} \\
\prod_{p\in S} C_{\mF_p}(\mF_p) \arrow[right hook->]{r}{\iota_S} & \prod_{p\in S}J_{\mF_p}(\mF_p)
\end{tikzcd}
$$
where $S$ is the set of primes of good reduction. 

Suppose that there exits some other non-singular point $\infty\neq P \in C(\mQ)$. The idea of the Mordell--Weil sieve is to derive a contradiction from various bits of local information coming from $\iota_S$, using the global constraint that a rational point on the curve maps into $J(\mQ)$. We explain the details in Section \ref{Gen3}, and refer the reader to \cite{bruin2010MWsieve} for a further discussion.

\subsection{Prym varieties}
Let $\pi\colon D\to C$ be an unramified finite morphism of degree $2$ between curves over $K$ and let $\iota\colon D\to D$ be the non-trivial involution of $D/C$. The Riemann--Hurwitz theorem implies that $g(C)>0$ and $g(D) = 2g(C)-1$. The associated \cdef{Prym variety} $\Prym(D/C)$ is the connected component containing 0 of the kernel $\pi_* \colon \Jac_D \rra \Jac_C$, which coincides with the image of $(\id_* - \iota_*)\colon \Jac_D \to \Jac_D$. Moreover, $\Prym(D/C)$ is an abelian subvariety of $\Jac_D$ of dimension $g(C) - 1$ with principal polarization coming from the restriction of the principal polarization on $\Jac_D$. Historically, Prym varieties provided examples of principally polarized abelian varieties, which are not Jacobian varieties.

In our situation, $C$ is a genus 3 non-hyperelliptic curve. 
Bruin \cite{bruin2008arithmetic} finds an explicit description of the associated Prym variety as $\Jac_F$ where $F$ is a genus 2 hyperelliptic curve. In addition to the description of the Prym variety, he gives an explicitly computable map $\varphi \colon D_{\delta} \rra \Jac_{F_{\delta}}$. 
Bruin's map does not require the existence of a rational point on $D_{\delta}$, so we could apply this construction to prove that $D_{\delta}(\mQ)$ is empty even if $D_{\delta}$ does have local points everywhere. In good circumstances, the rank of $\Jac_{F_{\delta}}(\mQ)$ is 0 for all relevant twists, and after finding the torsion subgroup of $\Jac_{F_{\delta}}(\mQ)$ and pulling back to $D_{\delta}(\mQ)$, we can determine the $\mQ$-points of $C$ by computing the image of $D_{\delta}(\mQ)$ under $\pi$. If the rank is positive, then one must proceed in a different manner.

\section{Analysis of Rational Points --- Genus 2}\label{sec:gen2}
There are 12 isomorphism classes of composite level modular curves with genus 2. Among these, 6 have Jacobians with rank 0, 4 with rank 1, and 2 with rank 2. We will use \'etale descent on the rank 2 cases and Chabauty and quotients on the others. In each case, the rank of the Jacobian is computed with \texttt{Magma}'s \texttt{RankBound} instrinsic. In the subsections below, the curve $X$ will denote a hyperelliptic curve of genus 2, and $J_X$ its Jacobian.

\subsection{Rank 0} If $\rank \, \Jac_{X}(\mQ) = 0$, then $\Jac_X(\mQ)$ is torsion. To find all of the rational points on $X$, it suffices to compute the torsion subgroup of $\Jac_X(\mQ)$ and compute the preimages under an Abel--Jacobi map $X \ext \Jac_X$. This is implemented in \texttt{Magma} as the \texttt{Chabauty0(J)} command, where $\texttt{J}$ is $\Jac_X$. 

\subsection{Rank 1}\label{subsec:rank1} If $\rank \, \Jac_X(\mQ) = 1$, then one can attempt Chabauty's method. This is implemented in \texttt{Magma} as the \texttt{Chabauty(ptJ)} command, where \texttt{ptJ} is a $\mQ$-point on $\Jac_X$ which generates $\Jac_X/\Jac_X[\tors]$. The intrinsic combines the Chabauty--Coleman method with the Mordell--Weil sieve to provably find the rational points on $X$. 

\subsection{Rank 2}\label{subsec:rank2}
If $\rank \, \Jac_X(\mQ) = 2$, then Chabauty's method does not apply; instead, we proceed with \'etale descent. In each case, the Jacobian of $X$ has a rational 2-torsion point. Thus, given a model
$$X\colon y^2 = f(x)$$
of $X$, $f$ factors as $f_1f_2$ where both polynomials are of positive, even degree, and $X$ admits \'etale double covers $C_d \rra X$, where the curves $C_d$ is given by 
$$
C_d \colon \left\lbrace\begin{array}{l}
dy_1^2 =  f_1(x) \\
dy_2^2  = f_2(x).
\end{array}\right.
$$
Let $S$ denote the set of bad places as in Section \ref{5.3}. By \'etale descent, every rational point on $X$ lifts to a rational point on $C_d(\mQ)$ for $d$ in the set divisors of primes in $S$, there multiples, and negations. The Jacobian of $C_d$ is isogenous to $\Jac_X \times E_d$, where $E_d$ is the Jacobian of the (possibly pointless) genus one curve $dy_2^2 = f_2(x)$ (where we assume that $\deg f_2 \geq \deg f_1$, so that $\deg f_2 \geq 3$). 

The two curves $X_{H_{40},\ref{G43}}(24)$ and $X_{H_{97},\ref{G43}}(24)$ are isomorphic to the rank 2 hyperelliptic curve
$$H\colon y^2 = 2x^6 + 2 = 2(x^2 + 1)(x^4 - x^2 + 1).$$
This curve admits \'etale covers by the genus 3 curves 
$$
C_d\colon \left\lbrace
\begin{array}{l}
dy_1^2 =  (x^2 + 1) \\
dy_2^2  = 2(x^4 - x^2 + 1)
\end{array}\right.
$$
for $d \in \brk{\pm 1, \pm2 ,\pm 3,\pm 6}$. We find that the genus 1 curves $ dy_2^2  = 2(x^4 - x^2 + 1)$ only have local points everywhere when $d = 2$. We compute that the curve $2y_1^2 =  (x^2 + 1)$ is isomorphic to $\mP_{\mQ}^1$ and the curve $2y_2^2  = 2(x^4 - x^2 + 1)$ is isomorphic to the rank 0 elliptic curve 
$$E\colon y^2 + 2xy = x^3 - 8x^2 + 12x.$$
The diagram
$$
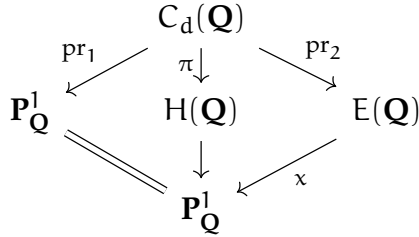
\begin{tikzcd}[row sep = 1em]
{} & C_d(\mQ) \arrow[swap]{dl}{\pr_1} \arrow{dr}{\pr_2} \arrow[swap]{d}{\pi}& {} \\
\mP_{\mQ}^1 \arrow[equal]{rd} & H(\mQ)\arrow{d}& E(\mQ) \arrow{ld}{x} \\
{} & \mP_{\mQ}^1 & {}
\end{tikzcd}
$$
tells us that the points on $C_d(\mQ)$ come from the preimages of the points on $E(\mQ)$. This allows us to determine the rational points on $C_d$ and thus on $H$ and on $X_{H_{40},\ref{G43}}(24)$ and $X_{H_{97},\ref{G43}}(24)$.

\section{Analysis of Rational Points --- Genus 3}\label{sec:gen3}
There are 15 isomorphism classes of genus 3 curves. 
Of these classes, 10 are hyperelliptic. 
The curves $X_{\ref{G223},\ref{G72}}(14)$ and $X_{\ref{G223},\ref{G76}}(14)$ are hyperelliptic and have rank equal to 0, and we handle these curves by using a Mordell--Weil sieve argument. 
The remaining hyperelliptic cases occur when considering composite-$(16,3)$ level modular curves, and we handle these cases using quotients or \'etale descent. 

The other 5 isomorphism classes $$X_{H_{105},\ref{G43}}(48),\,X_{H_{106},\ref{G43}}(48),\, X_{H_{107},\ref{G43}}(48),\,X_{H_{109},\ref{G43}}(48), \aad X_{H_{124},\ref{G43}}(48)$$ are non-hyperelliptic. 
The curve $X_{H_{109},\ref{G43}}(48)$ admits a rank 0 sub-quotient; by 
using Prym varieties, we determine the points on the rank 2 curve $X_{H_{124},\ref{G43}}(48)$; we also provably find the points on the rank 1 curve $X_{H_{106},\ref{G43}}(48)$ through a Chabauty argument.
For the remaining 2 isomorphism classes, we are unable to compute all of the $\mQ$-points, and we discuss our attempts in Section \ref{subsubsec:impish}.

\subsection{Analysis of genus 3 hyperelliptic curves} \label{Gen3}
As mentioned above, we find models for some composite level modular curves as genus 3 hyperelliptic curves. 

\subsubsection{Analysis of $X_{\ref{G223},\ref{G72}}(14)$} \label{Gen3}
The modular curve $X_{\ref{G223},\ref{G72}}(14)$ has a model given by the genus 3 hyperelliptic curve 
\begin{align*}
X_{\ref{G223},\ref{G72}}(14) \colon & y^2 =(x^3 - 4x^2 + 3x + 1)(x^4 - 10x^3 + 27x^2 - 10x - 27).
\end{align*}
For simplicity, we denote the smooth projective compactification of this modular curve by $X$.
\texttt{Magma} computes that $\rank \Jac_X(\mQ) = 0$, so $\Jac_X(\mQ)$ is torsion. We find that there exists a non-singular point $[1:0:0] \in X(\mQ)$, and we claim that this is in fact the only point on $X$. For ease of notation, we shall denote this point as $P_0$. 

From \cite[Exercise~C.4]{hindry2000diophantine}, we have $\#\Jac_{X_{\mF_p}}(\mF_p) = P_1(1)$, where $P_1(T)$ is the numerator of the Weil zeta function of $X_{\mF_p}(\mF_p)$ for some prime $p$. Moreover, by computing this value for a large number of primes and taking greatest common divisor, we find that $\# \Jac_X(\mQ)$ must divide $6.$ Since we have a non-singular point $P_0$ on $X$, we can embed $X$ into $\Jac_X$ via an Abel--Jacobi map
\begin{eqnarray*}
X(\mQ) &\ext & \Jac_X(\mQ) \\
P &\longmapsto & [P-P_0].
\end{eqnarray*}
Our above computation tells us the possible torsion in $\Jac_X(\mQ)$ is of order $1,2,3$ or $6$. Recall that the prime to $p\neq 2$ torsion of $\Jac_X(\mQ)$ injects into $\Jac_{X,{\mF_p}}(\mF_p)$. Let $S = \brk{5,11}$ and consider the Mordell--Weil sieve from Section~\ref{MWsieve}.

Suppose there exists another non-singular point $P\in X(\mQ)$. Since the divisor $[P-P_0]$ is a torsion point of $\Jac_X(\mQ)$, then it must also be torsion over $\mF_p$ for all $p$. Using \texttt{Magma}, we can enumerate $X_{\mF_p}(\mF_p)$ and check individually the orders of their respective images in $\Jac_{X,\mF_p}(\mF_p)$. We compute that the points on $X_{\mF_5}(\mF_5)$ map to points of exact order in $\brk{1,51}$ in $\Jac_{X,\mF_5}(\mF_5)$ and the points on $X_{\mF_{11}}(\mF_{11})$ map to points of exact order in $\brk{1, 8, 20, 40, 60, 120}$ in $\Jac_{X,\mF_{11}}(\mF_{11})$. Since none of these values, except for 1, coincide and the prime to $p$ torsion injects, we have that the possible orders of the divisor $[P-P_0]$ in $\Jac_X(\mQ)$ are either $5$ or $11$. However, our initial computation told us that the possible torsion in $\Jac_X(\mQ)$ must divide 6, and so this absurdity proves that $\brk{P_0} = X(\mQ)$.


\subsubsection{Analysis of $X_{H_{156},\ref{G43}}(48)$} \label{7.2}
The modular curve $X_{H_{156},\ref{G43}}(48)$ has a model as a genus 3 hyperelliptic curve
\begin{equation*}
X_{H_{156},\ref{G43}}(48)\colon y^2 = -x^7 - 8x.
\end{equation*} 
For simplicity, we denote the smooth projective compactification of this modular curve by $X$.
\texttt{Magma} computes that the rank of $X$ is at most 3. The rational points on $X$ lift to twists of the \'etale double cover by the genus 5 curves 
$$
C_d\colon \left\lbrace
\begin{array}{l}
dy_1^2 =  x \\
dy_2^2  = -(x^2+2)(x^4-2x^2+4)
\end{array}\right.
$$
for by $d \in \brk{\pm 1, \pm2 ,\pm 3,\pm 6}$. Each of these curves maps to the genus 2 hyperelliptic curve 
\begin{equation*}
H_d\colon dy^2  = -(x^2+2)(x^4-2x^2+4). 
\end{equation*}
For the above $d$, the Jacobian of $H_d$ has rank 1. Using quotients and Chabauty, we determine that there are 4 CM points on $X_{H_{156},\ref{G43}}(48)$ corresponding to $j=-3375$ and $j=16581375$.

\subsection{Analysis of genus 3 non-hyperelliptic curves}\label{subsec:g3super}
In this subsection, we analyze the rational points on the composite-$(16,3)$ level modular curves $X$ which have affine equation $y^3 = f(x^2)$.

\subsubsection{Analysis of $X_{H_{109},\ref{G43}}(48)$}
The modular curve $X_{H_{109},\ref{G43}}(48)$ is a genus 3, non-hyperelliptic curve with affine equation
$$X_{H_{109},\ref{G43}}(48)\colon y^3  = 4(x^4 - 8x^2 + 8).$$
For simplicity, we denote the smooth projective compactification of this modular curve by $X$. 
The canonical image of $X\subset \mP^2$ is the smooth plane quartic
$$C\colon -4v^4 + u^3w + 32v^2w^2 - 32w^4 = 0.$$
This curve has a two to one map to the elliptic curve
$$E\colon v^2 + 128v = u^3 - 2048,$$
which has rank zero with trivial torsion subgroup. To wit, we conclude that $X_{H_{109},\ref{G43}}(48)$ has no $\mQ$-rational points.

\subsubsection{Analysis of $X_{H_{124},\ref{G43}}(48)$}\label{Prym}
The modular curve $X_{H_{124},\ref{G43}}(48)$ is the genus 3 non-hyperelliptic curve with affine equation
$$X_{H_{124},\ref{G43}}(48)\colon y^3 =  2(x^4 + 4x^2 + 2)^2.$$
For simplicity, we denote the smooth projective compactification of this modular curve by $X$. 
We compute that $X$ maps to an elliptic curve $E$ with Mordell--Weil group $E(\mQ)\iso \mZ \oplus \ZZ{2}$. 
The existence of two torsion in $\Jac_X$ implies that $X$ admits an \'etale double cover.
By \cite{bruin2008arithmetic}, a genus 3 non-hyperelliptic curve over $\mQ$ admits an \'etale double cover if and only if it admits a model of the form 
$$Q_1(u,v,w)Q_3(u,v,w) = Q_2(u,v,w)^2$$
where $Q_1,Q_2,Q_3 \in \mQ[u,v,w]$ are quadratic forms. 
The canonical image of $X\subset \mP^2$ is the smooth plane quartic
$$C\colon u^4 + 4u^2v^2 + 2v^4 - 2vw^3 = 0$$
with determinantal decomposition 
$$
\begin{array}{l}
Q_1(u,v,w) \coloneqq  2vw + 2v^2, \\
Q_2(u,v,w) \coloneqq  u^2 + 2v^2, \\
Q_3(u,v,w) \coloneqq  v^2 - vw + w^2.
\end{array}
$$
From these, we construct a genus 5, unramified double cover $D_{\delta}$ by
$$D_{\delta}\colon \left\lbrace
\begin{array}{l}
Q_1(u,v,w) = \delta r^2 \\
Q_2(u,v,w) = \delta rs \\
Q_3(u,v,w) = \delta s^2.
\end{array}\right.$$
where $\delta \in \brk{\pm 1,\pm 2, \pm 3,\pm 6}$.  Let $\iota \colon [u:v:w:r:s] \mapsto [u:v:w:-r:-s] $ be an involution of $D_{\delta}$. Every point on $C(\mQ)$ lifts to 2 points on $D_{\delta}(\mQ)$ via $\iota$. Thus, in order to determine the rational points on $C$, it suffices to determine the rational points on $D_{\delta}$ for each $\delta$. 

Let $P_0 = [1:0:0:0:1] \in D_{\delta}(\mQ)$. We can embed $D_{\delta}$ in $\Jac_{D_{\delta}}$ via an Abel--Jacobi map
\begin{eqnarray*}
D_{\delta} & \ext & \Jac_{D_{\delta}} \\
P &\longmapsto & [P-P_0].
\end{eqnarray*}
When we compose this map with the projection map $(\id_* - \iota_*)\colon \Jac_{D_{\delta}} \rra \Prym(D_{\delta}/C)$, we obtain the Abel--Prym map 
\begin{eqnarray*}
D_{\delta} & \lrra & \Prym(D_{\delta}/C) \\
P &\longmapsto & [P-\iota(P)] - [P_0 - \iota(P_0)].
\end{eqnarray*}
 
Using the \texttt{Magma} code from \cite{bruin2008arithmetic},  we have the following diagram
$$
\begin{tikzcd}
D_{\delta} \arrow{r}{\varphi} \arrow{d} & \Jac_{F_{\delta}} \\
C & {} & 
\end{tikzcd}
$$
where $F_{\delta}$ is a genus 2 hyperelliptic curve. 
We find that every twist but the trivial one either has no real points or is not locally soluble at $2$ or $3$. When $\delta = 1$, we find that $\Jac_{F_{1}}(\mQ)$ has rank zero and torsion subgroup of size four. 
We compute that the four known points on $D_{\delta}(\mQ)$ map to distinct points in $\Jac_{F_1}(\mQ)$, and hence we deduce that the two known points on $C(\mQ)$ are in fact the only points. 
We conclude by checking that these points are cuspidal and CM corresponding to $j = 1728$.

\subsubsection{Analysis of $X_{H_{106},\ref{G43}}(48)$}\label{Chab}
As above, the modular curve $X_{H_{106},\ref{G43}}(48)$ is a genus 3 non-hyperelliptic curve with affine equation
$$X_{H_{106},\ref{G43}}(48)\colon y^3 = x^4 + 8x^2 + 8.$$
For simplicity, we denote the smooth projective compactification of this modular curve by $X$. 
We first attempt the above methods of quotients and Prym varieties. 
We find a non-trivial map from our curve $X$ to an elliptic curve $E$ with positive rank and $\ZZ{2}$ torsion, and so quotients do not yield a desired result.
As above, the existence of two torsion implies that our curve $X$ admits an \'etale double cover, and 
we compute the determinantal decomposition of $X$ as well as the double cover $D_{\delta}$. 
There exists a twist $\delta$ such that the Prym variety $\Prym(D_{\delta}/C)$ has positive rank, 
and so the above technique does not apply.

Using the \texttt{Magma} instrinsic
\begin{verbatim}
RankBound(x^4 + 8x^2 + 8,3);
\end{verbatim}
we compute that the rank of $X$ is at most $1$, which suggests that we proceed by a Chabauty argument. 
The canonical image of $X\subset \mP^2$ is the smooth plane quartic
$$C\colon u^4 + 8u^2v^2 + 8v^4 + vw^3=0.$$
We compute that $C(\mQ)$ contains two non-singular points $Q_0 \coloneqq [2:0:1]$ and $Q_1\coloneqq [1:0:0]$, and we claim that these are in fact the only points. 
By considering different reduction modulo $p$, we see that the $J_C(\mQ)[\tors] \subset \ZZ{2} \times \ZZ{2}$. 
Recall that for a genus 3 non-hyperelliptic curve $C$, the differences of bitangents of $C_{\overline{\mQ}}$ will generate $J_{C_{\overline{\mQ}}}[2]$.
Furthermore, by determining these differences, we see that there is only a single $2$ torsion point coming from the elliptic curve $E$, and so $J_C(\mQ)\iso \mZ \oplus \ZZ{2}$. 

Using the point $Q_1$, we can define an Abel--Jacobi map $C(\mQ)\ext J_C(\mQ)$ and form the degree zero divisor $D = [Q_0 - Q_1]$. 
Observe that 5 is a prime of good reduction for $C$, and consider the Chabauty setup:
$$
\begin{tikzcd}
C(\mQ) \arrow[right hook->]{r} \arrow[right hook->]{d} & C(\mQ_5) \arrow[right hook->]{d} \\
J_C(\mQ) \arrow[right hook->]{r} & J_C(\mQ_5).
\end{tikzcd}
$$
By considering the reduction mod 5, we see that the class $D$ is not divisible by two or three, and hence $D$ is a generates a finite index subgroup of $J_C(\mQ)$ since prime to $p\neq 2$ torsion injects \cite[Appendix]{katz_Galoisproperties}. 
We wish to find a differential $\omega_J \in H^0(J_{\mQ_5},\Omega^1)$ such that 
$$\int_0^D \omega_J = 0.$$

We see that $6D$ lies in the kernel of reduction modulo $5$ and that there are 6 points in $C_{\mF_5}(\mF_5)$; two of these $P_0 = \overline{Q_0}$ and $P_1= \overline{Q_1}$ lie in the image of the Mordell--Weil group under the Abel--Jacobi induced by $P_1$. 
Since $C$ is non-hyperelliptic, the linear system $|6D + 2[Q_1]|$ is either empty or zero-dimensional, and we verify that $|6D + 2[Q_1]| = D'$ where $D' = \Tr([34:2\sqrt{-86430}:225])$.
If we set $Q=[34: 2\sqrt{-86430}:225]$ and $Q'$ the conjugate, then we see that $Q$ and $Q_1 = [1:0:0]$ lie in the same residue disk $D_{P_1}$ of $C(K)$ for the ramified extension $K = \mQ_5(\sqrt{15})$. Moreover, we have
$$\int_0^{D}\omega_J = \frac{1}{6}\int_0^{6D} \omega_J = \frac{1}{6}\pwr{{\int_0^{[Q-Q_1]}\omega_J \, +\, \int_0^{[Q'-Q_1]}\omega_J}} = \frac{1}{6}\pwr{\int_{Q_1}^{Q}\omega_C \, +\, \int_{Q_1}^{Q'}\omega_C},$$
where we identify $H^0(C_{\mQ_5},\Omega^1)$ with $H^0(J_{\mQ_5},\Omega^1)$ via \cite[Proposition~2.1]{siksek2009chabauty}.

We compute a basis $\brk{\omega,\omega',\omega''}$ for $H^0(C_{\mQ_5},\Omega^1)$ such that 
$$\omega |_{D_{P_1}} \equiv 2t + 2t^3 + 2t^5 + \cdots \pmod 5,$$
in particular the expression is odd in the local coordinate $t = y$. 
Since $t(Q) = -t(Q')$ as $Q$ and $Q'$ are conjugate, we see that 
$$\int_{Q_1}^{Q}\omega \, + \, \int_{Q_1}^{Q'}\omega = 0,$$
and so $\omega\in \Lambda_C$. 
Moreover, the number of zeros for $\int\omega$ inside $D_{P_1}$ bounds the size of $\#(C(\mQ) \cap D_{P_1})$.
Standard Chabauty results (cf.~\cite[Section~2]{siksek2009chabauty}) assert that $\omega$ has one zero in the residue disk $D_{P_1}$, and 
since $5 > 2  + 1 + 1$, results of Stoll \cite[Lemma~6.1 \& Proposition~6.3]{stoll2006independence} imply that the number of zeros of $\int\omega$ within $D_{P_1}$ is bounded by $2$.
Therefore, we deduce that the $p$-adic integral $\int\omega$ only vanishes on the conjugate tuple $\brk{Q,Q'}$ and on the known rational point $Q_1 = [1:0:0]$ inside $D_{P_1}$, and so $C(\mQ)\cap D_{P_1} = Q_1$. 
A similar argument for $Q_0$ shows that $C(\mQ)\cap D_{P_0}= Q_0$.
The residue disks around $P_0$ and $P_1$ are the only relevant ones since these are the only points which lie in the image of the Abel--Jacobi $C_{\mF_5}(\mF_5)\ext J_{X,\mF_5}(\mF_5)$ determined by $P \mapsto [P-P_1]$. 
Furthermore, we conclude that $C(\mQ) = \brk{Q_0,Q_1}$.

\subsubsection{The impish ones}\label{subsubsec:impish}
There are two isomorphism classes 
\begin{align*}
X_{H_{105},\ref{G43}}(48) \colon &  y^3 = x^4 - 8x^2 + 8 ,\\
X_{H_{107},\ref{G43}}(48)\colon &  y^3 = 8x^4 - 8x^2 + 1
\end{align*}
of non-hyperelliptic genus 3 curves whose rational points we could not provably determine. We record our attempts here, discuss why the above methods do not work, and suggest further techniques for analysis. For the remainder of this section, let $X$ denote the smooth projective compactification of one of these modular curves.

First, these two isomorphism classes have Jacobians of rank at most 3, which strongly suggests that Chabauty's method is not possible. 
We next attempt an argument using Prym varieties. 
For each isomorphism class, we find the determinantal decomposition and form the \'etale double cover $D_{\delta}$. 
To our chagrin, each $X$ has a twist $F_{\delta}$ with positive rank. 
The curve $X_{H_{107},\ref{G43}}(48)$ has a twist $F_{\delta}$ with rank 1, and while Chabauty on the Prym is possible, the implement is difficult; see \cite[Section~8]{bruin2008arithmetic} for a ``by hand" example. 
The other curve $X_{H_{105},\ref{G43}}(48)$ has a twist with rank 2 meaning we cannot attempt Chabauty's method on the Prym. 
These curves could admit other \'etale double covers coming from non-trivial $2$-torsion in $J_X(\mQ)$. However, we determine that $J_X(\mQ)[\tors] \iso \ZZ{2}\times \ZZ{4}$ through local considerations, and so each curve $X$ only admits one such double cover.

As above, we find that these curves map to a positive rank elliptic curve, and so $J_X \sim E \times A$ where $A$ is some abelian surface over $\mQ$. 
We first attempt to decompose $A$ into sub-factors and hope that we find a rank zero piece.
We determine that $A$ is simple over $\mQ$ using Honda--Tate theory and computing that the characteristic polynomial of Frobenius is irreducible for some prime $p$.
Similar computations strongly suggest that $A$ splits over a quadratic extension $K/\mQ$ into the the two-fold product of an elliptic curve $E_2$ with good reduction outside of 2 and 3.  
In order to find this elliptic curve, we need to enumerate the non-isomorphic elliptic curves over $K$ with such reduction.

Thankfully, a theorem of Shafarevich \cite[Theorem~6.1]{silvermanAEC} states that there is a finite list of such elliptic curves over any number field $K/\mQ$. Cremona and Lingham \cite{cremona2007finding} give an explicit algorithm for finding such curves, which involves computing the integral points on a particular set of elliptic curves. This procedure has been implemented in \texttt{Magma} as the command 
\begin{verbatim}
EllipticCurveWithGoodReductionSearch(2*3*O,500);
\end{verbatim}
(see \cite{MR1484478} for the documentation). To our chagrin, we do not find our desired elliptic curves over any quadratic extension ramified at $2$ and/or $3$. If one did find a quadratic extension $K$ and the elliptic curve $E_2/K$ as above, then one could proceed with elliptic Chabauty, a technique pioneered by Bruin \cite{bruin2003chabauty}. However, this technique is not fully implemented in \texttt{Magma} since one needs to construct a map from the curve $X_{K}$ to $E_2$, which may or may not come from a quotient mapping. 
To conclude, we check that the known points are CM and/or cuspidal and conjecture that there are no other $\mQ$-points on these curves.


\section{Analysis of Rational Points --- Higher genus}\label{sec:HighGen}
In our computations, we find models for our composite level modular curve of genus greater than 3. There are 3 genus 6 hyperellipitc curves (2 isomorphism classes $X_{H_{171},\ref{G43}}(48)$ and $X_{H_{172},\ref{G43}}(48)$). These 2 curves have rank equal to 0, and we handle them by finding explicit generators for the torsion subgroup of the Jacobian. 
We encounter 1 genus 7 curve coming from the anomalous genus 1 modular curve $X_{\ref{G223}}(11)$ with infinitely many points. We also come across 8 genus 4 non-hyperelliptic curves whose construction resembles that of the genus 7 curve. In some cases, we can easily find the rational points on these curves using quotients. However, there are 2 isomorphism classes of such curves whose $\mQ$-points we cannot provably determine. Finally, there are 3 curves with unknown, large genera that come from the three outstanding cases mentioned in List \ref{list13}. 

\subsection{Analysis of genus 4 non-hyperelliptic curves}
There are 8 non-hyperelliptic curves of genus 4 occurring as composite-$(16,3)$ level modular curves $X_{H_{n},\ref{G43}}(48)$ where $n = 149,150,151,153,$ $160,161,165,166$. For $n=149,151,160,161$, the modular curve $X_{H_n}$ is isomorphic to a rank 0 elliptic curve; hence by computing pre-images, we easily find the points on our composite level modular curve. We cannot determine the $\mQ$-points on the following 2 isomorphism classes (represented via their canonical image in $\mP^3$)
$$
\begin{array}{ll}
X_{H_{150},\ref{G43}}\colon\left\lbrace \begin{array}{ll}
AC + 3BC - D^2 \\
A^2B - 2AB^2 - 7B^3 - C^3,
\end{array}\right. 
& X_{H_{153},\ref{G43}}\colon\left\lbrace \begin{array}{ll}
AC - BC - D^2 \\
A^2B + 2AB^2 - B^3 - 65536C^3,
\end{array}\right. 
\end{array}
$$
but we discuss our attempts below.


Let $X\coloneqq X_{H_n,\ref{G43}}(48)$ be one of the 2 remaining isomorphism classes defined above.
By construction, the curve $X$ is a cover of the rank 1 elliptic curve $E_1 \coloneqq X_{H_n}$. We also find that $X$ covers another non-isogenous elliptic curve $E_2$ of rank 1. Computations of local zeta functions at $p=7$ assert that
$$\Jac_{X} \sim E_1 \times E_2 \times A$$
where $A$ is a simple abelian surface. Similar computations of local zeta functions at $p^2$ suggest that $A$ is not geometrically simple and splits over some quadratic extension of $\mQ$.

We encounter the similar issues with these curves as we did in Section \ref{subsubsec:impish}. 
The best possible approach is elliptic Chabauty, but we were unsuccessful in finding elliptic curves $E',\, E''$ defined over a quadratic number field $K$ such that $A_K\sim E' \times E''$. 
As before, the hardest part of the implementation is finding the morphism from $X_K$ to $E'$ or $E''$ defined over $K$.  
Another approach is to construct an \'etale double cover of these curves.
To proceed, one first shows that part of the 2-torsion in $J_X(\mQ)$ comes from an elliptic factor $E$.
Thus, one can form the normalization of the fibered diagram $X\times_{E_1} E_3$ with $\varphi^{\vee}\colon E_3\to E_1$ the dual isogeny to $\varphi\colon E_1 \to E_3$ with kernel the known 2-torsion point of $E_1$. 
This construction produces our double cover $Z\to X$ with non-optimal equations. 
Furthermore, working on the double cover does not ameliorate the original issue. 
As above, we check that the known rational points correspond to CM and/or cuspidal points.

\subsection{Analysis of genus 6 hyperelliptic curves}\label{Gen6}
There are two isomorphism classes of genus 6 hyperelliptic curves. The hyperelliptic curves
\begin{align*}
X_{H_{171},\ref{G43}}(48) \colon & y^2 = -x^{13} + 64x, \\
X_{H_{172},\ref{G43}}(48) \colon & y^2 = -x^{13} - 64x 
\end{align*}
are representatives for these classes.
\texttt{Magma} computes that the rank of each of these curves is 0, and so we proceed by computing the torsion subgroup of the Jacobian for each respective curve. 
By evaluating Weil zeta functions and comparing invariant factor decompositions of $J_{X,\mF_p}$, we conclude that $$J_X(\mQ) \iso (\ZZ{2})^6.$$ 
Since the two torsion of a hyperelliptic Jacobian is determine by the roots of the defining equation for the curve, we can find generators for each part of $J_X(\mQ)$ and conclude that these curves only possess a point at infinity and another corresponding to $[0 : 0 : 1]$. 

\subsection{Analysis of $X_{\ref{G223},\ref{G113}}(22)$}\label{Gen7}
There is only one modular curve from \cite{zywinapossible} of genus 1 with rank 1, namely $X_{G_3}(11)$, which is isomorphic to the elliptic curve $\sE \colon y^2 + y = x^3 - x^2 - 7x + 10$. In the appendix (cf.~Section~\ref{list11}), we recall the morhpism $J(x,y)$ corresponding to the map from $\sE \rra \mA_{\mQ}^1 \cup \brk{\infty}$ and that $\sE(\mQ) \iso \langle (4,5) \rangle$.  The composite-$(2,11)$ level modular curve 
\begin{equation*}\label{L22}
X_{\ref{G223},\ref{G113}}(22)\colon 
\left\lbrace\begin{array}{l}
y^2 + y = x^3 - x^2 - 7x + 10 \\
s^2 +12^3 = J(x,y)
\end{array} \right.
\end{equation*}
is a genus 7 non-hyperelliptic curve in $\mA^3_{\mQ}.$ For simplicity, we denote the smooth compactification of this curve by $X$. By pulling back points from $\sE(\mQ)$, we find a cuspidal point and the CM point $[x:y:s:z] = [2:0:0:1]$ on $X$. Unfortunately, we are unable to provably compute the rational points on the curve $X$. Below, we discuss the attempted techniques and facts about said curve. 

We know that 
$$\Jac_X \backsim \sE \times A$$
where $\sE$ is the elliptic curve defined above and $A$ is some $6$-dimensional abelian variety. Since $\rank \sE = 1$, we want $A$ to decompose in some way; ideally, we would want $A$ to have rank 0 some elliptic factor. Empirical evidence suggests that $A$ is not simple over $\mQ$ and that $A$ is isogenous to $A_1\times A_2\times A_3$ where $A_i$ are abelian surfaces. However, we are not able to determine the genus 2 curves $C_i$ whose Jacobians $J_i$ are isomorphic to $A_i$.

\begin{remark}
Jeremy Rouse has written \texttt{Magma} code which "guesses" how the Jacobian of a modular curve $X$ decomposes by comparing points counts of $X$ with traces of $a_f(p)$, where $f$ is a newform of level $p$, and he ran this code on the composite-$(2,11)$ level modular curve $X_{\ref{G223},\ref{G113}}(22)$. His results suport that $A$ decomposes as above, but also that each $A_i$ has analytic rank 0! Unfortunately, the genus 2 curves whose Jacobians are isomorphic to $A_i$ are not in the LMFDB database of genus 2 curves \cite{lmfdb}, but this does give evidence that there are no non-obvious points on $X_{\ref{G223},\ref{G113}}(22)$. 
\end{remark}


Following another suggestion of Jeremy Rouse, the geometry of the curve suggests that we search for sub-curves of $X$. If there exists a curve $C$ such that $X\to C$, then there \textit{could} exist some applicable subgroup $H$ such that  $\ref{G223} \times \ref{G113} \subseteq H \subset \GL_2(\ZZ{22})$, which witnesses $C$ as the modular curve associated to $H$. We compute the list of such subgroups and the genera of the associated modular curves to deduce that the only modular quotient is a $\mP^1_{\mQ}$ comes from the unique maximal subgroup $H'$ of $\GL_2(\ZZ{22})$ containing $\ref{G223} \times \ref{G113} $. 
The quotient of $X$ by its automorphism group $\ZZ{2}$ produces the know elliptic sub-curve $E$. Although these techniques did not produce a sub-curve, they do not entirely rule out the possibility of a map from $X$ to a curve of lower genus. We conjecture the following.

\begin{conjecture}
There does not exist a non-CM elliptic curve $E$ over $\mQ$ with square discriminant such that $(\rho_{E,2} \times \rho_{E,11})(G_{\mQ})$ is simultaneously non-surjective. 
\end{conjecture}

\subsection{The cursed ones}\label{subsec:cursed}
Up to conjugacy, there are 4 maximal subgroups of $\GL_2(\ZZ{13})$ that have surjective determinant, namely $\ref{G136},N_{\spl}(13),N_{\nsp}(13)$ and $\ref{G137}$. Zywina handles the cases concerning the subgroups of $\ref{G136}$, and the other three subgroups correspond to the outstanding cases. 

Baran \cite{baran2014exceptional} showed that the modular curves $X_{N_{\spl}}(13)$ and $X_{N_{\nsp}(13)}$ are both isomorphic to the genus 3 curve $C$ defined in $\mP_{\mQ}^2$ with equation
$$(y-z)x^3 + (2y^2 + zy)x^2 + (-y^3 + zy^2 -2z^2y + z^3)x + (2z^2y^2 - 3z^3y) = 0.$$
Baran also gives the morphism from the above model to the $j$-line. The 7 known rational points on $C$ all correspond to cusps and CM points on $X_{N_{\spl}}(13)$ and $X_{N_{\nsp}}(13)$. 
Recently, Balakrishnan, Dogra, M\"uller, Tuitman, and Vonk \cite{balakrishnanetal_Splitcartan13} proved that $C$ has no other rational points, using explicit Chabauty--Kim methods. Their result is equivalent to saying that there does not exist a non-CM elliptic curve over $\mQ$ with $\rho_{E,13}(G_{\mQ})$ conjugate to a subgroup of $N_{\spl}(13)$ and $N_{\nsp}(13)$. 

Banwait and Cremona \cite{banwait2014tetrahedral} have shown that $X_{\ref{G137}}(13)$ is isomorphic to the genus 3 curve $C'$ defined in $\mP_{\mQ}^2$ with equation 
$$
\begin{small}
\begin{array}{c}
4x^3y - 3x^2y^2 + 3xy^3 - x^3z + 16x^2yz - 11xy^2z + 5y^3z + 3x^2z^2 + 9xyz^2 + y^2z^2 + xz^3 + 2yz^3 = 0.\end{array}\end{small}$$
The authors also give the morphism from the modular curve to the $j$-line. The 4 known rational points on $C'$ correspond to a CM points and three non-CM points. Conjecturally, $C'$ has no other rational points, which is equivalent to saying that \cite[Theorem~1.8(iv)]{zywinapossible} gives necessary and sufficient condition on when $\rho_{E,13}(G_{\mQ})$.

We check that: the points on $C$ do not pull back to points $X_{\ref{G223},N_{\spl}}(26)$, the point $[0:0:1]$ on C pulls back to the CM point corresponding to $j=0$ on $X_{\ref{G223},N_{\nsp}}(26)$, and the known points on $C'$ do not pullback to points on $X_{\ref{G223},\ref{G137}}(26)$. Following the above conjectures, we formulate our own concerning the composite-$(2,13)$ image of Galois.
\begin{conjecture}
There does not exist a non-CM elliptic curve $E$ over $\mQ$ with square discriminant such that $(\rho_{E,2} \times \rho_{E,13})(G_{\mQ})$ is simultaneously non-surjective. 
\end{conjecture}

\section{Applications ---  Entanglement Fields}\label{AppEnt}
In this final section, we discuss applications of our results to the study of entanglement fields. 
An elliptic curve $E$ over $K$ has \cdef{$(m_1,m_2)$-entanglement fields} if $K(E[m_1]) \cap K(E[m_2]) \neq K$ for some positive integers $m_1,\, m_2$. 

In this scenario, ``most" elliptic curves over $\mQ$ have \textit{quadratic} $(2,n)$-entanglement fields for some $n\in \mZ_{>0}$. Indeed, elliptic curves with square discriminant form a thin set in the sense of Serre, so ``most" elliptic curves have non-square discriminant.
For these curves, the $2$-division field $\mQ(E[2])$ will contain $\mQ(\sqrt{\Delta_E})$. By Kronecker--Weber, there exists some $n$ such that $\mQ(\sqrt{\Delta_E})$ is contained in $\mQ(\zeta_n)$, and the Weil paring implies that $\mQ(\zeta_n) \subset \mQ(E[n])$. Therefore, these curves satisfy $\mQ(E[2]) \cap \mQ(E[n]) \supseteq \mQ(\sqrt{\Delta_E})$, and so ``most" elliptic curves have quadratic entanglement fields.

In light of this fact, we restrict our consideration to non-CM elliptic curves over $\mQ$ with entanglement fields $\mQ(E[\ell_1^{m_1}]) \, \cap \, \mQ(E[\ell_2^{m_2}]) \neq \mQ$ for distinct primes $\ell_1$, $\ell_2$ and positive integers $m_1,\, m_2$.
Note that this condition corresponds to the phenomena of the $(\ell_1^{m_1},\ell_2^{m_2})$-composite level image of Galois being contained in a proper subgroup of $\rho_{E,\ell_1^{m_1}}(G_{\mQ}) \times \rho_{E,\ell_2^{m_2}}(G_{\mQ})$.

\subsection{Statement of results}
Using Theorem \ref{2.1}, we prove existence results for $(2,5)$ and $(2,7)$ entanglement fields of degree three when $E$ has square discriminant.
We also exhibit an infinite family of elliptic curves over $\mQ$ with $(2,p^n)$-entanglement fields of degree three where $3\mid p-1$. 
Finally, we complete the classification of non-abelian $(2,3)$-entanglement fields first studied by Brau and Jones \cite{braujones2014elliptic}. For the remainder of this section, we ignore elliptic curves with rational full 2-torsion since these curves cannot have $(2,n)$-entanglement fields by definition.

An important tool in our study of entanglements is Goursat's topological lemma (see \cite[Lemma~5.2.1]{ribet1976galois} for proof). 
\begin{lemma}[Goursat's lemma]\label{Goursat}
Let $G_0$ and $G_1$ be groups and $G\subseteq G_0 \times G_1$ a subgroup satisfying 
$$\pi_i(G) = G_i \quad (i \in \brk{0,1}),$$
where $\pi_i$ denotes the canonical projection onto the $i\tth$-factor. Then there exists a normal group $Q$ and surjective homomorphisms $\psi_0 \colon G_0 \rra Q$, $\psi_1 \colon G_1\rra Q $\, for which $$G = \brk{(g_0,g_1) \in G_0 \times G_1 : \psi_0(g_0) = \psi_1(g_1)}.$$
\end{lemma} 
The idea is to use our results concerning composite level modular curves to find possibilities for entanglement. Then we apply Goursat's lemma to sift out the cases where entanglement cannot occur from a group theoretic viewpoint. From here, we compute division fields using \texttt{Magma} and check for entanglements. To demonstrate the technique, we first present a result proving the lack of entanglement fields for a family of elliptic curves over $\mQ$.

\begin{lemma}\label{8.2}
Let $E$ be a non-CM elliptic curve over $\mQ$ with square discriminant. Then $\mQ(E[2]) \, \cap \, \mQ(E[5]) = \mQ$.
\end{lemma}
\begin{proof}
From Theorem~\ref{2.1} and Proposition \ref{3.2}, we know that there is only one possibility for non-surjective composite-$(2,5)$ image, namely the image is conjugate to $\ref{G223}\times \ref{G59}$. The subgroup $\ref{G59}$ do not contain an index 3 normal subgroup, hence Lemma \ref{Goursat}, implies that there does not exists a subgroup $G \leq \GL_2(\ZZ{2}) \times \GL_2(\ZZ{5})$ that projects onto the mod 2 and mod 5 image. Therefore, these curves cannot have entanglement fields via the Galois correspondence.
\end{proof}

\subsection{$(2,7)$-entanglement fields}\label{8.3}
From Theorem~\ref{2.1}, the only possibility for simultaneous non-surjective composite-$(2,7)$ image of Galois is $\ref{G223} \times \ref{G77} \leq \GL_2(\ZZ{2}) \times \GL_2(\ZZ{7})$. The subgroup $G_7$ does contain an index 3, normal subgroup, so the points on the modular curve $X_{\ref{G223},\ref{G77}}(14)$ correspond to $j$-invariants of elliptic curves with possible entanglement fields coming from $\mQ(E[2]) \, \cap \, \mQ(E[7])$. Since such an elliptic curve $E$ has $7$-division field of degree 252, it is computationally inefficient to study the subfields of $\mQ(E[7])$ or even $\mQ(x(E[7]))$, where the latter number field contains the $x$-coordinates of the $7$-torsion points. Hence, in order to perform computations, we need to find a subfield of $\mQ(E[7])$ with manageable degree.

Since $Z(\GL_2(\ZZ{7})) \leq G_7$ and \,$\# Z(\GL_2(\ZZ{7})) = 6$, the fixed field $L \coloneqq \mQ(E[7])^{Z(\GL_2(\ZZ{7}))}$ is an index 6 subfield of $\mQ(E[7])$. From \cite[Table~5.1]{adelmann2001decomposition}, $L$ is a degree 42 number field defined by the $7\tth$-modular polynomial $\Phi_7(X,j_E)$. For non-CM $E$ coming from $X_{\ref{G223},\ref{G77}}(14)$, we compute degree 3 subfields of $L$ and check whether they are isomorphic to $\mQ(E[2])$; below, we give two examples of non-CM elliptic curves with mod 2 and mod 7 entanglement fields.

\begin{exam}
The non-CM elliptic curves
\begin{align*}
E_1 \colon &y^2 + xy = x^3 - 4/129825457969x - 1/1168429121721 \\
E_2 \colon & y^2 + xy = x^3 - 4/2209x - 1/19881
\end{align*}
have $(2,7)$-entanglement fields of degree three.
\end{exam}

\subsection{$(2,p^n)$-entanglement fields}
In \cite{rubin2001mod}, Rubin and Silverberg give explicit equations for elliptic curves over a field of characteristic $\neq 2,3$ with prescribed mod 2 image of Galois. By constructing an elliptic curves over $\mQ$ with special 2-division field, we exhibit an infinite family of elliptic curves with $(2,p^n)$-entanglement of degree three.

\begin{prop}
Let $p$ be a prime $\geq 7$ such that $3\mid p-1$ and $n$ a positive integer. Then there exist infinitely many elliptic curves over $E/\mQ$ such that $\mQ(E[2])\,\cap\,\mQ(E[p^n]) \iso L$, where $L$ is the degree 3 number field of $\mQ(\zeta_{p^n})$ by our assumption on $p$. Furthermore, we give an explicit parametrization of such elliptic curves.
\end{prop}

\begin{proof}
Since $\varphi(p^n) = p^{n-1}(p-1)$ where $\varphi$ is the Euler-totient function, the cyclotomic field $\mQ(\zeta_{p^n})$ contains the degree 3 intermediate field $L$ of $\mQ(\zeta_p)$. Gauss \cite{gauss1966disquisitiones} showed that the minimal polynomial of $L$ is 
\begin{equation*}
g(X) = X^3 + X^2 + (p-1)X/3 - ((p-1)/3 + kp)/9
\end{equation*}
where $k$ is uniquely determined by the integral representation $4p = (3k - 2)^2 + 27N^2$. 

Let $E$ be the elliptic curve with defining polynomial $g(X)$. Using the change of variables $(X,Y) \rra (x-1/3,y)$, we find a Weierstrass model of the form
$$E\colon y^2 = x^3 - \frac{p}{3}x + \frac{p(2-3k)}{27}.$$
The construction of $E$ forces the 2-division field $\mQ(E[2])$ to be isomoprhic to $L$. Using \cite[Theorem~1.1]{rubin2001mod}, the elliptic curve 
\begin{align*}
\mE_t \colon y^2 = & \, 
 x^3 +  \frac{(1727pt^2 + p + 9/4k^2t^2 - 9/4k^2 - 3kt^2 + 3k + t^2 - 1)}{(p - 9/4k^2 + 3k - 1)}x \\ 
&  \tiny{\frac{(-1727pt^3 - 5181pt^2 + 3pt + p - 9/4k^2t^3 - 27/4k^2t^2)}{(p - 9/4k^2 + 3k - 1)}} + \\
& \frac{(-27/4k^2t - 9/4k^2 +
    3kt^3 + 9kt^2 + 9kt + 3k - t^3 - 3t^2 - 3t - 1)}{(p - 9/4k^2 + 3k - 1)}
\end{align*}
has $2$-torsion subgroup isomorphic to that of $E$ for $t \in \mQ$. The existence of the Weil pairing implies that elliptic curves of the form $\mE_t$ satisfy $\mQ(\mE_t[2]) \cap \mQ(\mE_t[p^n]) \iso L$.
\end{proof}

\begin{remark}
Above, we present the general equation for $\mE_t$. For a specific prime $p$ and unique $k$, the defining equation for $\mE_t$ can be quickly computed using the \texttt{Magma} instrinsic \texttt{RubinSilverbergPolynomials(2,j)}, where $j$ is the $j$-invariant of the elliptic curve $E$. 
\end{remark}

\subsection{$(2,3)$-entanglement fields}
In a recent work \cite{braujones2014elliptic}, Brau and Jones exhibit a modular curve of level 6 over $\mQ$ whose $\mQ$-rational points correspond to $j$-invariants of elliptic curves $E$ over $\mQ$ with $\mQ(E[2]) \subseteq \mQ(\zeta_3,\Delta_E^{1/3}) \subseteq \mQ(E[3])$ and hence $(2,3)$-entanglement fields. The construction of their modular curve begins with finding the unique index 6 normal subgroups $\cN \leq \GL_2(\ZZ{3})$ defined by
$$\cN \coloneqq \brk{\begin{psmallmatrix} x & -y \\ y & x \end{psmallmatrix} : x^2 + y^2 \equiv 1 \mod 3 } \sqcup  \brk{\begin{psmallmatrix} x & y \\ y & -x \end{psmallmatrix} : x^2 + y^2 \equiv -1 \mod 3}.$$

The authors observe that $\cN$ fits into the exact sequence
$$
\begin{tikzcd}
1 \arrow{r} & \cN \arrow[right hook->]{r} & \GL_2(\ZZ{3}) \arrow[two heads]{r}{\theta} & \GL_2(\ZZ{2}) \arrow{r} & 1.
\end{tikzcd}
$$
Their modular curve of level 6 corresponds to the subgroup $H' \leq \GL_2(\ZZ{6})$ coming from the graph of $\theta$
\begin{align*}
H' &\coloneqq \brk{(g_2,g_3) \in \GL_2(\ZZ{2}) \times \GL_2(\ZZ{3}) : g_2 = \theta(g_3)}.
\end{align*}
The points lying in the image of $j(X_{H'})$ correspond to $j$-invariants of elliptic curves over $\mQ$ satisfying the above division field condition. The surjectivity of $\theta$ tells us that such elliptic curves have surjective mod 2 image of Galois as well. We summarize and slightly clarify their results concerning these curves in the following theorem.

\begin{theorem}[\protect{\cite[Theorem~1.4]{braujones2014elliptic}}]\label{8.7}
Let $E$ be a non-CM elliptic curve over $\mQ$ with $j$-invariant of the form 
$$j_E = 2^{10}3^3t^3(1 - 4t^3)$$
where $t \in \mQ\setminus\brk{0,1/2}$. This family of elliptic curves has surjective mod 2 image of Galois and non-abelian entanglement fields $$\mQ(E[2]) \, \cap \, \mQ(E[3]) \iso \mQ(\zeta_3,\Delta_E^{1/3}).$$
\end{theorem}

Brau and Jones pose the question \cite[Question~1.1]{braujones2014elliptic} of classifying the triples $(E,m_1,m_2)$ with $E$ an elliptic curve over a number field $K$ and $m_1,m_2$ a pair of relatively prime integers for which the mod $m_1$ and mod $m_2$ entanglement field is non-abelian over $K$. We ask whether the elliptic curves defined in Theorem \ref{8.7} are the only ones with non-abelian  $(2,3)$-entanglement fields. By constructing covers of the modular curve $X_{H'}$, we find another family of elliptic curves with such entanglement fields and provide a complete answer to \cite[Question~1.1]{braujones2014elliptic} in the case where $K = \mQ$ and $(m_1,m_2) = (2,3)$.

\begin{theorem}\label{8.8}
There exist infinitely many non-CM elliptic curves $E$ over $\mQ$ with composite-$(2,3)$ image of Galois conjugate to $\GL_2(\ZZ{2}) \times \ref{G33}$ and non-abelian entanglement fields $$\mQ(E[2]) \, \cap \, \mQ(E[3]) \iso \mQ(\zeta_3,\Delta_E^{1/3}) .$$
In particular, these curves either satisfy
$$\mQ(E[2]) \, \cap \, \mQ(E[3]) \iso  \mQ(\zeta_3,\Delta_E^{1/3})\iso \mQ(x(E[3])),$$
or
$$\mQ(E[2]) \, \cap \, \mQ(E[3]) \iso  \mQ(\zeta_3,\Delta_E^{1/3}) \iso \mQ(E[3]).$$
Furthermore, we provide a parametrization of such elliptic curves.
\end{theorem}

\begin{proof}
Let $\ref{G33}$ be the level 3 applicable subgroup from List \ref{list3}, which we can identify as the Borel subgroup of $\GL_2(\ZZ{3})$. There exists a unique index 6 normal subgroup of $\ref{G33}$, namely 
$$G' \coloneqq \left\langle \begin{psmallmatrix} 2 & 0 \\ 0 & 2 \end{psmallmatrix} \right\rangle.$$
Since $G' \leq \cN$, we have the following exact sequences
$$
\begin{tikzcd}
1 \arrow{r} & \cN \arrow[right hook->]{r} & \GL_2(\ZZ{3})   \arrow[two heads]{r}{\theta_1} & \GL_2(\ZZ{2}) \arrow{r} \arrow[equal]{d} & 1 \\
1 \arrow{r} & G' \arrow[right hook->]{r} \arrow[right hook->]{u} & G_3 \arrow[two heads]{r}{\theta_2} \arrow[right hook->]{u}& \GL_2(\ZZ{2}) \arrow{r} & 1.
\end{tikzcd}
$$
Let 
\begin{align*}
H'' &\coloneqq \brk{(g_2,g_3) \in \GL_2(\ZZ{2}) \times \ref{G33} : g_2 = \theta_2(g_3)}
\end{align*}
denote the graph of $\theta_2$. Since $H''\leq H'$, there is a map between the modular curves $X_{H''} \rra X_{H'}$. Using List \ref{list3}, we can construct the level 6 modular curve corresponding to $H''$, namely
$$X_{H''}\colon 2^{10}3^3s^3(1 - 4s^3) = \frac{27(t+1)(t+9)^3}{t^3}.$$
This a genus 0 curve endowed with a rational point, hence isomorphic to $ \mP_{\mQ}^1$. The rational points on the curve $X_{H''}$ correspond to elliptic curves over $\mQ$ with composite-$(2,3)$ image conjugate to a subgroup of $\GL_2(\ZZ{2}) \times \ref{G33}$ and $\mQ(E[2]) \subseteq \mQ(\zeta_3,\Delta_E^{1/3}) \subseteq \mQ(E[3])$.

The surjectivity of $\theta_2$ implies that the mod 2 image is surjective, and hence the conditions on $\rho_{E,3}(G_{\mQ})$ imply that the 3-division field has degree at least 12 and contains $\mQ(E[2]) \subset \mQ(\zeta_3,\Delta_E^{1/3})$. Therefore, the rational points on $X_{H''}$ classify elliptic curves with non-abelian entanglement $$\mQ(E[2]) \, \cap \, \mQ(E[3]) \iso \mQ(\zeta_3,\Delta_E^{1/3}).$$
When $\rho_{E,3}(G_{\mQ}) $ is conjugate to $ \ref{G33}$ (equivalently when the composite-$(2,3)$ image surjects onto $\ref{G33}$), we have
$$\mQ(E[2]) \, \cap \, \mQ(E[3]) \iso \mQ(\zeta_3,\Delta_E^{1/3})\iso \mQ(x(E[3])),$$
and when $\rho_{E,3}(G_{\mQ}) $ is conjugate to $ \ref{H331}$ or $\ref{H332}$, then 
$$\mQ(E[2]) \, \cap \, \mQ(E[3]) \iso \mQ(\zeta_3,\Delta_E^{1/3})\iso \mQ(E[3]).$$
The parametrization for these curves can be found at \cite{JSM2017}.
\end{proof}

\begin{coro}\label{8.9}
There do not exist non-CM $j$-invariants outside of those from Theorems \ref{8.7} and \ref{8.8} corresponding to elliptic curves with non-abelian $(2,3)$-entanglement fields.
\end{coro}

\begin{proof}
The only applicable subgroups of level 3 that have an index 6, normal subgroup are $\ref{G33}$, $\ref{H331}$, and $\ref{H332}$ where the latter two subgroups are index two subgroups of the first. In particular, \cite[Theorem~1.2]{zywina10maximal} asserts that elliptic curves with such mod 3 images have the same $j$-invariant; curves with the latter two images of Galois contain rational $3$-torsion whereas the first does not. 
Since the $S_3$ is the only non-abelian mod 2 image of Galois, our result follows from Lemma~\ref{Goursat}.
\end{proof}
\clearpage
\appendix
\section{Applicable prime level subgroups} \label{App:AppendixA}
In this appendix, we reproduce the list of applicable subgroups from \cite{zywinapossible} and give the rational function expressing the modular $j$-invariant for certain exceptional primes $\ell$ in addition to presenting these subgroups as lattices in $\GL_2(\ZZ{\ell})$ for easy navigation. We decorate the cases where we cannot provably analyze the rational points on the modular curve $X_{\ref{G223},G_{n,\ell}}(2\cdot \ell)$ with a tilde. Also the subgroups are hyperlinked to their definition.

\subsection{\underline{List}$(\ell = 2)$}\label{list2}
Up to conjugacy there are three proper subgroups of $\GL_2(\ZZ{2})$, all of which are arithmetically maximal:
\begin{align*}
\customlabel{G221}{G_{1,2}} &= \brk{I}, 
& \customlabel{G222}{G_{2,2}}& = \brk{I,\begin{psmallmatrix} 1 & 1 \\ 0 & 1\end{psmallmatrix}}, 
& \customlabel{G223}{G_{3,2}}& = \brk{I,\begin{psmallmatrix} 1 & 1 \\ 1 &0\end{psmallmatrix},\begin{psmallmatrix}0 & 1 \\ 1 & 1\end{psmallmatrix}}.
\end{align*}
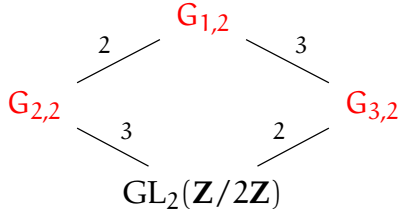
\begin{figure}[h!]
$$
\begin{tikzcd}[column sep = 1em, row sep = 1em]
{} & \ref{G221} & {} \\
\ref{G222} \arrow[dash]{ru}{2} &{}& \ref{G223}\arrow[dash, swap]{lu}{3} \\
{} & \GL_2(\ZZ{2}) \arrow[dash]{ru}{2} \arrow[dash, swap]{lu}{3} & {}\\
\end{tikzcd}
$$
\caption{Applicable subgroup lattice for $\GL_2(\ZZ{2})$}
\end{figure}

 From \cite[Theorem~1.1]{zywinapossible}, $\rho_{E,2}(G_{\mQ})$ is conjugate in $\GL_2(\mF_{2})$ to a subgroup of $G_i$ if and only if $j_E$ is of the form:
\begin{align*}
J_1(t) &= 256 \frac{(t^2 + t+ 1)^3}{t^2(t+1)},  
&J_2(t) &= 256 \frac{(t+ 1)^3}{t}, 
 &J_3(t) &= t^2 + 1728
\end{align*}
for some $t\in \mQ$ and each respective $i$.

\subsection{\underline{List}$(\ell = 3)$}\label{list3}
Define the following subgroups of $\GL_2(\ZZ{3})$:
\begin{enumerate}
\item[$\bullet$] let $\customlabel{G13}{G_{1,3}}$ be the group $C_{\spl}(3)$,  
\item[$\bullet$] let $\customlabel{G23}{G_{2,3}}$ be the group $N_{\spl}(3)$,
\item[$\bullet$] let $\customlabel{G33}{G_{3,3}}$ be the group $B(3)$,
\item[$\bullet$] let $\customlabel{G43}{G_{4,3}}$ be the group $N_{\nsp}(3)$,
\item[$\bullet$] let $\customlabel{H311}{H_{\brk{1,1},3}}$ be the subgroup consisting of the matrices of the form $\begin{psmallmatrix} 1 & 0 \\ 0 & *\end{psmallmatrix}$,
\item[$\bullet$] let $\customlabel{H331}{H_{\brk{3,1},3}}$ be the subgroup consisting of the matrices of the form $\begin{psmallmatrix} 1 & * \\ 0 & * \end{psmallmatrix}$,
\item[$\bullet$] let $\customlabel{H332}{H_{\brk{3,2},3}}$ be the subgroup consisting of the matrices of the form $\begin{psmallmatrix} * & * \\ 0 & 1 \end{psmallmatrix}$.
\end{enumerate}
Each of the groups $G_i$ contain $-I$, and the groups $H_{\brk{i,j},3}$ do not contain $-I$. Moreover, we have $G_{i,3} = \pm H_{\brk{i,j},3}$. 
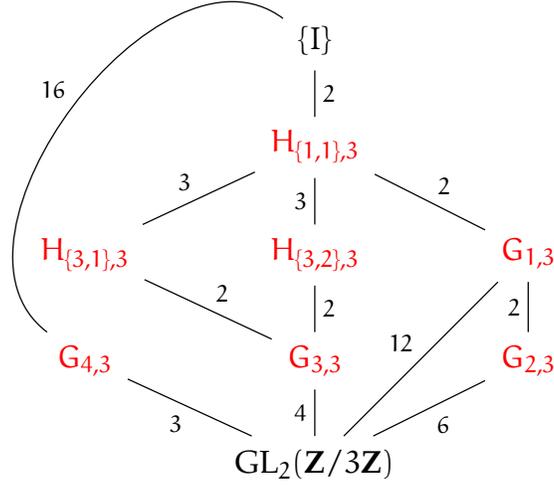
\begin{figure}[h!]
$$
\begin{tikzcd}[row sep = 1.5em]
{} & {} & \brk{I} & {} & {} \\
{} & {}& \ref{H311}\arrow[dash,swap]{u}{2}& {}& {} \\
{} & \ref{H331}\arrow[dash, crossing over]{ur}{3} &  \ref{H332}\arrow[dash]{u}{3}  & \ref{G13} \arrow[dash, swap]{lu}{2} & {} \\
{} &\ref{G43}  \arrow[dash, bend left = 90]{ruuu}{16}& \ref{G33}  \arrow[dash, swap]{u}{2} \arrow[dash, swap, crossing over]{ul}{2} & \ref{G23} \arrow[dash]{u}{2} & {} \\
{} & {} & \GL_2(\ZZ{3}) \arrow[dash, swap]{ru}{6} \arrow[dash]{ul}{3} \arrow[dash]{ruu}{12} \arrow[dash]{u}{4}& {} & {}
\end{tikzcd}
$$
\caption{Applicable subgroup lattice for $\GL_2(\ZZ{3})$}
\end{figure}

From \cite[Theorem~1.2(ii)]{zywinapossible}, $\rho_{E,3}(G_{\mQ})$ is conjugate in $\GL_2(\mF_{3})$ to a subgroup of $G_{i,3}$ if and only if $j_E$ is of the form:
\begin{align*}
J_1(t) &= 27\frac{(t+1)^3(t+3)^3(t^2+3)^3}{t^3(t^2+3t+3)^3}, 
&J_2(t) &= 27\frac{(t+1)^3(t-3)^3}{t^3}, \\
J_3(t) &= 27\frac{(t+1)(t+9)^3}{t^3},  
&J_4(t) &= t^3
\end{align*}
for some $t \in \mQ$ and each respective $i$. Furthermore, \cite[Theorem~1.2(iii,iv)]{zywinapossible} provides explicit conditions (isomorphisms) when $\rho_{E,3}(G_{\mQ})$ is conjugate to $H_{\brk{i,j},3}$ for $i=1,3$ and $j = 1,2$.

\subsection{\underline{List}$(\ell = 5)$}\label{list5}
Define the following subgroups of $\GL_2(\ZZ{5})$:
\begin{enumerate}
\item[$\bullet$] let $\customlabel{G51}{G_{1,5}}$ be the subgroup consisting of the matrices of the form $\pm \begin{psmallmatrix}1 & 0\\ 0 & * \end{psmallmatrix}$,
\item[$\bullet$] let $\customlabel{G52}{G_{2,5}}$ be the group $C_{\spl}(5)$,
\item[$\bullet$] let $\customlabel{G53}{G_{3,5}}$ be the unique subgroup of $N_{\nsp}(5)$ of index 3; it is generated by $\begin{psmallmatrix}2 & 0\\ 0 & 2 \end{psmallmatrix}, \begin{psmallmatrix}1 & 0\\ 0 & -1\end{psmallmatrix}$, and $\begin{psmallmatrix}0 & 1\\ 3 & 0 \end{psmallmatrix}$,
\item[$\bullet$] let $\customlabel{G54}{G_{4,5}}$ be the group $N_{\spl}(5)$,
\item[$\bullet$] let $\customlabel{G55}{G_{5,5}}$ be the subgroup consisting of the matrices of the form $\pm \begin{psmallmatrix}* & *\\ 0 & 1 \end{psmallmatrix}$,
\item[$\bullet$] let $\customlabel{G56}{G_{6,5}}$ be the subgroup consisting of the matrices of the form $\pm \begin{psmallmatrix}1& *\\ 0 & * \end{psmallmatrix}$,
\item[$\bullet$] let $\customlabel{G57}{G_{7,5}}$ be the group $N_{\nsp}(5)$,
\item[$\bullet$] let $\customlabel{G58}{G_{8,5}}$  be the group $B(5)$,
\item[$\bullet$] let $\customlabel{G59}{G_{9,5}}$ be the unique maximal subgroup of $\GL_2(\ZZ{5})$ which contains $N_{\spl}(5)$; it is generated by $\begin{psmallmatrix}2 & 0\\ 0 & 1 \end{psmallmatrix},\begin{psmallmatrix}1 & 0\\ 0 & 2 \end{psmallmatrix},\begin{psmallmatrix}0 & -1\\1 & 0 \end{psmallmatrix}$, and $\begin{psmallmatrix}1 & 1\\ 1 & -1 \end{psmallmatrix}$,
\item[$\bullet$] let $\customlabel{H511}{H_{\brk{1,1},5}}$ be the subgroup consisting of the matrices of the form $\begin{psmallmatrix}1 & 0\\ 0 & * \end{psmallmatrix}$,
\item[$\bullet$] let $\customlabel{H512}{H_{\brk{1,2},5}}$ be the subgroup consisting of the matrices of the form $\begin{psmallmatrix}a^2 & 0\\ 0 & a \end{psmallmatrix}$,
\item[$\bullet$] let $\customlabel{H551}{H_{\brk{5,1},5}}$be the subgroup consisting of the matrices of the form $\begin{psmallmatrix}* & *\\ 0 & 1 \end{psmallmatrix}$,
\item[$\bullet$] let $\customlabel{H552}{H_{\brk{5,2},5}}$ be the subgroup consisting of the matrices of the form $\begin{psmallmatrix}a & *\\ 0 & a^2 \end{psmallmatrix}$,
\item[$\bullet$] let $\customlabel{H561}{H_{\brk{6,1},5}}$ be the subgroup consisting of the matrices of the form $\begin{psmallmatrix}1 & *\\ 0 & * \end{psmallmatrix}$,
\item[$\bullet$] let $\customlabel{H562}{H_{\brk{6,2},5}}$ be the subgroup consisting of the matrices of the form $\begin{psmallmatrix}a^2 & *\\ 0 & a \end{psmallmatrix}$.
\end{enumerate}
Each of the groups $G_i$ contain $-I$, and the groups $H_{\brk{i,j},5}$ do not contain $-I$. Moreover, we have $G_{i,5} = \pm H_{i,j}$. 
\begin{figure}[h!]
$$
\begin{tikzcd}[row sep=3.5em, column sep=4.5em]
{} & {} & \brk{I} & {} & {} \\
{} & {} & {} & \ref{H511}\arrow[dash]{lu}{4} & \ref{H512} \arrow[dash, swap]{llu}{4}\\
\ref{H551}\arrow[dash]{uurr}{20} & \ref{H552} \arrow[dash, swap, near start]{uur}{20} & \ref{H561} \arrow[dash]{uu}{20}  \arrow[dash, ]{ur}{5} &\ref{H562} \arrow[dash]{uul}{} \arrow[dash, near end]{ur}{5}& \ref{G51} \arrow[dash]{u}{2} \arrow[dash, near start]{lu}{ 2}\\
\ref{G53}\arrow[dash, bend left = 90]{rruuu}{16} & \ref{G55}\arrow[dash]{u}{ 2} \arrow[dash, crossing over]{lu}{2}& \ref{G56}\arrow[dash]{u}{2} \arrow[dash]{ru}{ 2}& {}& \ref{G52}\arrow[dash]{u}{2}\\
{}& {} & \ref{G58}\arrow[dash]{lu}{2} \arrow[dash]{u}{2} \arrow[dash]{ruur}{ 10}& {} & \ref{G54} \arrow[dash]{u}{2} \\
\ref{G57}\arrow[dash]{uu}{3} & {} & {} & {} & \ref{G59}  \arrow[dash]{u}{3} \\
{}& {} & \GL_2(\ZZ{5}) \arrow[dash]{uu}{6} \arrow[dash]{llu}{10} \arrow[dash, swap]{rru}{5} \arrow[dash, swap]{rruu}{15}& {} & 
\end{tikzcd}
$$
\caption{Applicable subgroup lattice for $\GL_2(\ZZ{5})$}
\end{figure}
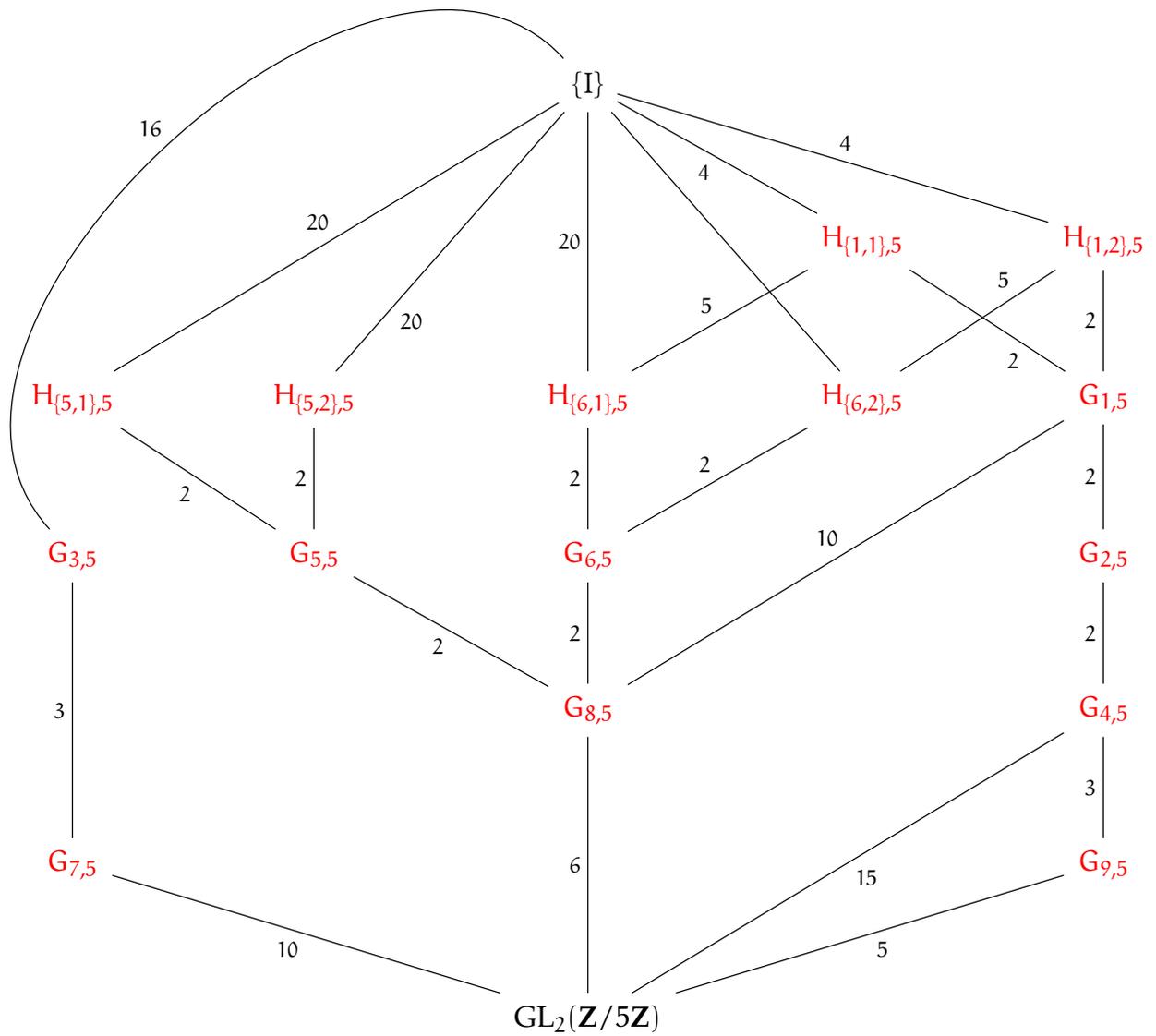

From \cite[Theorem~1.4(ii)]{zywinapossible}, $\rho_{E,5}(G_{\mQ})$ is conjugate in $\GL_2(\mF_{5})$ to a subgroup of  $G_{i,5}$ if and only if $j_E$ is of the form:
\begin{align*}
J_1(t) & =  \frac{(t^{20} + 228t^{15} + 494t^{10} - 228t^5 + 1)^3}{t^5(t^{10} - 11t^5  - 1)^5},  \\ 
J_2(t) & =  \frac{(t^2 + 5t + 5)^3(t^4+5t^2 + 25)^3(t^4 + 5t^3 + 20t^2 + 25t + 25)^3}{t^5(t^4 + 5t^3 + 15t^2 + 25t + 25)^5},\\ 
J_3(t) & =  \frac{5^4t^3(t^2 + 5t + 10)^3(2t^2 + 5t + 5)^3(4t^4 + 30t^3 + 95t^2 + 150t + 100)^3}{(t^2 +5t+5)^5(t^4 +5t^3 +15t^2 +25t+25)^5} ,\\
J_4(t) & =   \frac{(t+5)^3(t^2-5)^3(t^2+5t+10)^3}{(t^2+5t+10)^3}, \\
J_5(t) & =  \frac{ (t^4 + 228t^3 + 494t^2 - 228t + 1)^3}{t(t^2 - 11t -1)^5},\\
J_6(t) & =   \frac{(t^4 - 12t^3 + 14t^2 + 12t + 1)^3}{t^5(t^2-11t -1)},\\
J_7(t) & =  \frac{5^3(t+1)(2t+1)^3(2t^3-3t+3)^3}{(t^2 + t -1)^5},\\
J_8(t) & =  \frac{5^2(t^2 + 10t + 5)^3}{t^5},\\
J_9(t) & =  t^3(t^2+5t+40)
\end{align*}
for some $t\in \mQ$ and each respective $i$. Furthermore, \cite[Theorem~1.4(iii)]{zywinapossible} provides explicit conditions when $\rho_{E,5}(G_{\mQ})$ is conjugate to $H_{\brk{i,j},5}$ for $i=1,5,6$ and $j=1,2$.

\subsection{\underline{List}$(\ell = 7)$}\label{list7}
Define the following subgroups of $\GL_2(\ZZ{7})$:
\begin{enumerate}
\item[$\bullet$] let $\customlabel{G71}{G_{1,7}}$ be the subgroup of $N_{\spl}(7)$ consisting of elements of $C_{\spl}(7)$ with square determinant and elements of $N_{\spl}(7)\setminus C_{\spl}(7)$ with non-square determinant; it is generated by $\begin{psmallmatrix} 2 & 0 \\ 0 & 4 \end{psmallmatrix}, \begin{psmallmatrix} 0& 2 \\ 1 & 9\end{psmallmatrix}, $ and $\begin{psmallmatrix} -1 & 0 \\ 0 & -1\end{psmallmatrix}$,
\item[$\bullet$] let $\customlabel{G72}{G_{2,7}}$ be the group $N_{\spl}(7)$,
\item[$\bullet$] let $\customlabel{G73}{G_{3,7}}$ be the subgroup consisting of matrices of the form $\pm \begin{psmallmatrix} 1 & * \\ 0  &* \end{psmallmatrix}$,
\item[$\bullet$] let $\customlabel{G74}{G_{4,7}}$ be the subgroup consisting of matrices of the form $\pm \begin{psmallmatrix} * & * \\0 & 1\end{psmallmatrix}$,
\item[$\bullet$] let $\customlabel{G75}{G_{5,7}}$ be the subgroup consisting of matrices of the form $\begin{psmallmatrix} a & * \\ 0 & \pm a \end{psmallmatrix}$,
\item[$\bullet$] let $\customlabel{G76}{G_{6,7}}$ be the group $N_{\nsp}(7)$,
\item[$\bullet$] let $\customlabel{G77}{G_{7,7}}$ be the group $B(7)$,
\item[$\bullet$] let $\customlabel{H711}{H_{\brk{1,1},7}}$ be the subgroup generated by $\begin{psmallmatrix} 2 & 0 \\ 0 & 4\end{psmallmatrix}$ and $\begin{psmallmatrix}0 & 2 \\ 1 & 0 \end{psmallmatrix}$,
\item[$\bullet$] let $\customlabel{H731}{H_{\brk{3,1},7}}$ be the subgroup consisting of matrices of the form $\begin{psmallmatrix} 1& * \\0 & * \end{psmallmatrix}$,
\item[$\bullet$] let $\customlabel{H732}{H_{\brk{3,2},7}}$ be the subgroup consisting of matrices of the form $\begin{psmallmatrix} \pm 1 & * \\ 0 & a^2\end{psmallmatrix}$,
\item[$\bullet$] let $\customlabel{H741}{H_{\brk{4,1},7}}$ be the subgroup consisting of matrices of the form $\begin{psmallmatrix} * & * \\ 0 & 1\end{psmallmatrix}$,
\item[$\bullet$] let $\customlabel{H742}{H_{\brk{4,2},7}}$ be the subgroup consisting of matrices of the form $\begin{psmallmatrix}a^2 & * \\ 0 & \pm 1 \end{psmallmatrix}$,
\item[$\bullet$] let $\customlabel{H751}{H_{\brk{5,1},7}}$ be the subgroup consisting of matrices of the form $\begin{psmallmatrix} \pm a^2 & * \\ 0 & a^2\end{psmallmatrix}$,
\item[$\bullet$] let $\customlabel{H752}{H_{\brk{5,2},7}}$ be the subgroup consisting of matrices of the form $\begin{psmallmatrix} a^2 & * \\ 0 & \pm a^2\end{psmallmatrix}$,
\item[$\bullet$] let $\customlabel{H771}{H_{\brk{7,1},7}}$be the subgroup consisting of matrices of the form $\begin{psmallmatrix} * & * \\ 0 & a^2\end{psmallmatrix}$,
\item[$\bullet$] let $\customlabel{H772}{H_{\brk{7,2},7}}$ be the subgroup consisting of matrices of the form $\begin{psmallmatrix} a^2 & * \\ 0 & *\end{psmallmatrix}$. 
\end{enumerate}
Each of the groups $G_i$ contain $-I$, and the groups $H_{\brk{i,j},7}$ do not contain $-I$. Moreover, we have $G_{i,7} = \pm H_{\brk{i,j},7}$. 
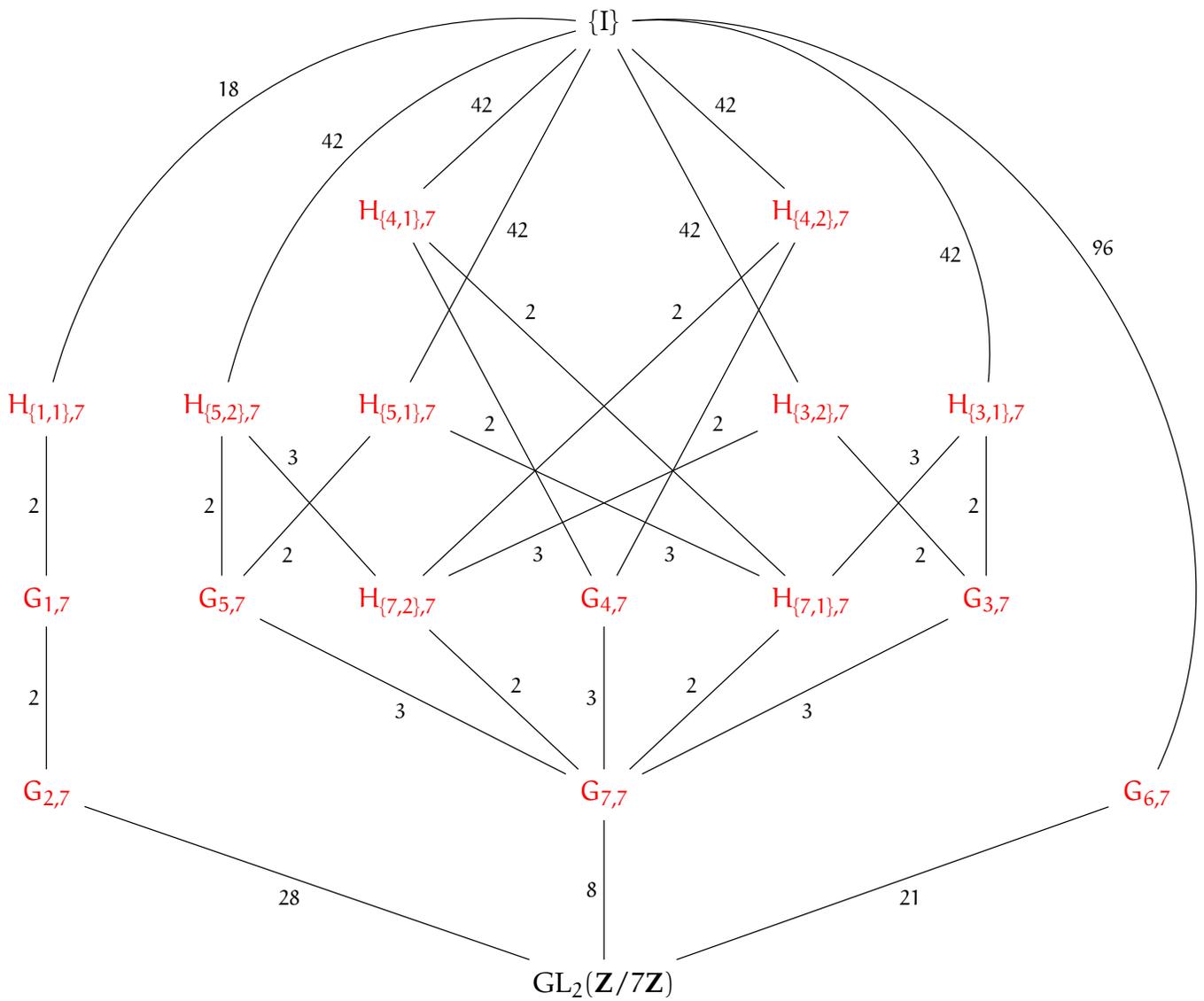
\begin{figure}[h!]
$$
\begin{tikzcd}[back line/.style={densely dotted},column sep = 2.5em, row sep = 5em]
{} & {} & {} & \brk{I} & {} & {} & {} \\
{} & {} & \ref{H741}\arrow[dash]{ur}{42}& {} & \ref{H742}\arrow[dash, swap]{ul}{42}& {} & {} \\
\ref{H711}\arrow[dash, bend left = 400]{uurrr}{18}& \ref{H752}\arrow[dash, bend left = 30 ]{uurr}{42}& \ref{H751}  \arrow[dash, swap]{uur}{42}& {} &  \ref{H732} \arrow[dash]{uul}{42}& \ref{H731} \arrow[dash, bend right = 50, near start]{uull}{42} {} &\\
\ref{G71} \arrow[dash]{u}{2}& \ref{G75} \arrow[dash]{u}{2}\arrow[dash, near start, swap]{ur}{2} & \ref{H772}\arrow[dash, near end, swap]{lu}{3} \arrow[dash, near start, swap ]{rru}{3}  \arrow[dash, near end]{uurr}{2}& \ref{G74}\arrow[dash, swap]{uur}{2}\arrow[dash,  ]{uul}{2}& \ref{H771} \arrow[dash, near end]{ru}{3} \arrow[dash, near start ]{llu}{3}  \arrow[dash, near end,swap]{uull}{2}& \ref{G73} \arrow[dash, near start]{lu}{2}\arrow[dash]{u}{2}  & {} \\
\ref{G72} \arrow[dash]{u}{2}& {}& {} & \ref{G77} \arrow[swap,dash]{rru}{3} \arrow[dash]{u}{3} \arrow[dash]{llu}{3} \arrow[dash]{ru}{2} \arrow[dash, swap]{lu}{2}& {} & {} & \ref{G76}\arrow[dash, bend right = 60, swap]{uuuulll}{96}\\
{} & {} & {} & \GL_2(\ZZ{7}) \arrow[dash, swap]{rrru}{21} \arrow[dash]{lllu}{28} \arrow[dash]{u}{8} & {} & {} & {} 
\end{tikzcd}
$$
\caption{Applicable subgroup lattice for $\GL_2(\ZZ{7})$}
\end{figure}

From \cite[Theorem~1.5(ii)]{zywinapossible}, $\rho_{E,7}(G_{\mQ})$ is conjugate in $\GL_2(\mF_{7})$ to a subgroup of  $G_{i,7}$ if and only if $j_E$ is of the form:
\begin{align*}
J_1(t) &= 3^3\cdot 5 \cdot 7^5/2^7, \\
J_2(t) &= \frac{t(t+1)^3(t^2 - 5t - 1)^3(t^2 - 5t + 8)^3(t^4 - 5t^3 + 8t^2 - 7t + 7)^3}{(t^3 - 4t^2 + 3t + 1)^7}, \\
J_3(t) &= \frac{(t^2 - t + 1)^3(t^6 - 11t^5 + 30t^4 - 15t^3- 10t^2 + t + 1)^3}{(t-1)^7t^7(t^3 - 8t^2 + 5t + 1)}, \\
J_4(t) &= \frac{(t^2 - t + 1)^3(t^6 + 229t^5 + 270t^4 - 1695t^3 + 1430t^2 - 235t + 1)^3}{(t-1)t(t^3-8t^2+5t+1)^7}, \\
J_5(t) &= -\frac{(t^2 - 3t - 3)^3(t^2 - t + 1)^3(3t^2 - 9t + 5)^3(5t^2 - t- 1)^3}{(t^3-2t^2-t+1)(t^3-t^2-2t+1)^7}, \\
J_6(t) &= \frac{64t^3(t^2+7)^3(t^2-7t+14)^3(5t^2-14y-7)^3}{(t^3-7t^2+7t+7)^7}, \\
J_7(t) &= \frac{(t^2+245t+2401)^3(t^2+13t+49)}{t^7}
\end{align*}
for some $t\in \mQ$ and each respective $i$. Furthermore, \cite[Theorem~1.5(iii,iv)]{zywinapossible} provides us with explicit conditions when $\rho_{E,7}(G_{\mQ})$ is conjugate to $H_{\brk{i,j},7}$ for $i=1,3,4,5,7$ and $j = 1,2$.

\subsection{\underline{List}$(\ell = 11)$}\label{list11}
Define the following subgroups of $\GL_2(\ZZ{11})$:
\begin{enumerate}
\item[$\bullet$] let $\customlabel{G111}{G_{1,11}}$ be the subgroup generated by $\pm \begin{psmallmatrix} 1 & 1 \\ 0 & 1\end{psmallmatrix}$ and $\begin{psmallmatrix} 4 & 0 \\ 0 & 6\end{psmallmatrix}$,
\item[$\bullet$] let $\customlabel{G112}{G_{2,11}}$ be the subgroup generated by $\pm \begin{psmallmatrix} 1 & 1 \\ 0 & 1\end{psmallmatrix}$ and $\begin{psmallmatrix} 5 & 0 \\ 0 & 7\end{psmallmatrix}$,
\item[$\bullet$] let $\customlabel{G113}{G_{3,11}}$  be the group $N_{\nsp}(11)$,
\item[$\bullet$] let $\customlabel{H1111}{H_{\brk{1,1},11}}$ be the subgroup generated by $ \begin{psmallmatrix} 1 & 1 \\ 0 & 1\end{psmallmatrix}$ and $\begin{psmallmatrix} 4 & 0 \\ 0 & 6\end{psmallmatrix}$,
\item[$\bullet$] let $\customlabel{H1112}{H_{\brk{1,2},11}}$ be the subgroup generated by $ \begin{psmallmatrix} 1 & 1 \\ 0 & 1\end{psmallmatrix}$ and $\begin{psmallmatrix} 7 & 0 \\ 0 & 5\end{psmallmatrix}$,
\item[$\bullet$] let $\customlabel{H1121}{H_{\brk{2,1},11}}$ be the subgroup generated by $ \begin{psmallmatrix} 1 & 1 \\ 0 & 1\end{psmallmatrix}$ and $\begin{psmallmatrix} 5 & 0 \\ 0 & 7\end{psmallmatrix}$,
\item[$\bullet$] let $\customlabel{H1122}{H_{\brk{2,2},11}}$ be the subgroup generated by $ \begin{psmallmatrix} 1 & 1 \\ 0 & 1\end{psmallmatrix}$ and $\begin{psmallmatrix} 6 & 0 \\ 0 & 4\end{psmallmatrix}$.
\end{enumerate}
Each of the groups $G_i$ contain $-I$, and the groups $H_{\brk{i,j},11}$ do not contain $-I$. Moreover, we have $G_{i,11} = \pm H_{\brk{i,j},11}$. 
\begin{figure}[h!]
$$
\begin{tikzcd}[column sep = 2.5em, row sep = 3em]
{} & {} & {} & \brk{I} & {} & {} \\
{} & \ref{H1121} \arrow[ dash]{urr}{110}& \ref{H1122} \arrow[dash, near start, swap]{ur}{110}& {} & \ref{H1111} \arrow[ dash, near start]{ul}{110}&\ref{H1112}  \arrow[ dash, swap]{ull}{110}\\ 
{}&{}&  \ref{G112}\arrow[dash, near end, swap]{ul}{2}\arrow[dash, near end]{u}{2}& \widetilde{\ref{G113}} \arrow[ dash]{uu}{240}& \ref{G111}\arrow[dash, near end, swap]{ur}{2}\arrow[dash, near end]{u}{2}&  \\
{} & {}& {} & \GL_2(\ZZ{11})  \arrow[dash]{u}{55}  \arrow[dash, swap]{ru}{60} \arrow[dash]{lu}{60}& {} & {} 
\end{tikzcd}
$$
\caption{Applicable subgroup lattice for $\GL_2(\ZZ{11})$}
\end{figure}
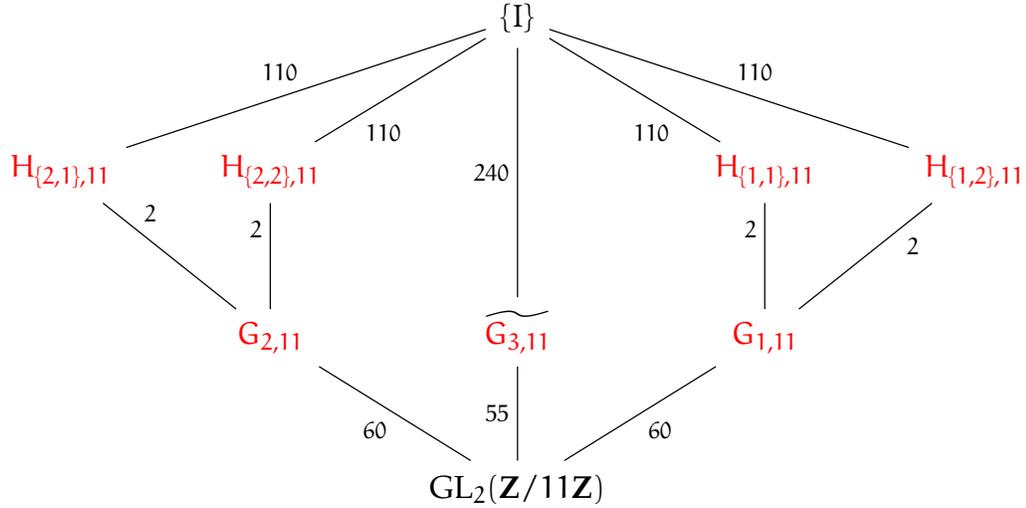

From \cite[Theorem~1.6(ii,iii)]{zywinapossible}, there are unique values for $j_E$ that correspond to $\pm \rho_{E,11}(G_{\mQ})$ being conjugate in $\GL_2(\ZZ{11})$ to a subgroup of $G_{1,11}$ and $G_{2,11}$. 

The modular curve $X_{G_3}(11) = X_{\nsp}^+(11)$ is the only one from Zywina's classification that has genus 1 with infinitely many rational points. To explicitly describe this modular curve, let $\sE$ be the elliptic curve over $\mQ$ defined by the Weierstrass equation $y^2 + y  = x^3 - x^2 - 7x + 10$ and let $\cO$ be the point at infinity. The Mordell--Weil group $\sE(\mQ)$ is an infinite cyclic group generated by the point $(4,5)$. Halberstadt \cite{halberstadt1998courbe} showed that $X_{\nsp}^+(11)$ is isomorphic to $\sE$ and that the morphism to the $j$-line corresponds to 
$$J(x,y) \coloneqq \frac{(f_1f_2f_3f_4)^3}{f_5^2f_6^{11}},$$
where 
\begin{align*}
f_1 &= x^2 + 3x - 6,  &f_2 &= 11(x^2 - 5y) + (2x^4 + 23x^3 - 72x^2 - 28x + 127), \\
f_3 &= 6_y + 11x - 19,  &f_4 &= 22(x-2)y + (5x^3 + 17x^2 - 112x - 120), \\
f_5 &= 11y + (2x^2 + 17x -34),  &f_6 &= (x-4)y - (5x- 9).
\end{align*}
From \cite[Theorem~1.6(iv)]{zywinapossible}, $\rho_{E,11}(G_{\mQ})$ is conjugate to $G_{3,11}$ if and only if $j_E = J(P)$ for some point $P \in \sE(\mQ)\setminus\brk{\cO}$. 
\begin{remark}
In \cite[Section~4.5.5]{zywinapossible}, Zywina gives explicit polynomials $A,B,C \in \mQ[x]$ of degree 55 such that for a non-CM elliptic curve $E/\mQ$, we have $j_E = J(P)$ for some $P \in \sE(\mQ) \setminus\brk{\cO}$ if and only if the polynomial $A(x)j_E^2 + B(x)j_E + C(x) \in \mQ[x]$ has a rational root. Hence given a numerical $j_E$, this gives a straightforward way to check the criterion that $\rho_{E,11}(G_{\mQ})$ is conjugate to a subgroup of $G_{3,11}$. 
\end{remark}

\subsection{\underline{List}$(\ell = 13)$}\label{list13}
Define the following subgroups of $\GL_2(\ZZ{13})$:
\begin{enumerate}
\item[$\bullet$] let $\customlabel{G131}{G_{1,13}}$ be the subgroup consisting of matrices of the form $\begin{psmallmatrix}* & * \\ 0 & b^3 \end{psmallmatrix}$,
\item[$\bullet$] let $\customlabel{G132}{G_{2,13}}$ be the subgroup consisting of matrices of the form $\begin{psmallmatrix}a^3 & * \\ 0 & *\end{psmallmatrix}$,
\item[$\bullet$] let $\customlabel{G133}{G_{3,13}}$ be the subgroup consisting of matrices $\begin{psmallmatrix}a & * \\ 0 & b\end{psmallmatrix}$ for which $(a/b)^4 = 1$,
\item[$\bullet$] let $\customlabel{G134}{G_{4,13}}$ be the subgroup consisting of matrices of the form $\begin{psmallmatrix}* & * \\ 0 & b^2 \end{psmallmatrix}$,
\item[$\bullet$] let $\customlabel{G135}{G_{5,13}}$ be the subgroup consisting of matrices of the form $\begin{psmallmatrix}a^2 & * \\ 0 & *\end{psmallmatrix}$,
\item[$\bullet$] let $\customlabel{G136}{G_{6,13}}$ be the group $B(13)$,
\item[$\bullet$] let $\customlabel{G137}{G_{7,13}}$ be the subgroup generated by the matrices $\begin{psmallmatrix} 2 & 0 \\ 0 & 2 \end{psmallmatrix}, \begin{psmallmatrix} 2& 0 \\ 0 & 3 \end{psmallmatrix}, \begin{psmallmatrix} 0 & -1 \\ 1 & 0 \end{psmallmatrix}$, and 
$\begin{psmallmatrix} 1& 1 \\ -1 & 1\end{psmallmatrix}$; it contains the scalar matrices and its image in $\PGL_2(\ZZ{13})$ is isomorphic to $\mathfrak{S}_4$,
\item[$\bullet$] let $\customlabel{H1341}{H_{\brk{4,1},13}}$ be the subgroup consisting of matrices of the form $\begin{psmallmatrix}* & * \\ 0 & a^4\end{psmallmatrix}$,
\item[$\bullet$] let $\customlabel{H1342}{H_{\brk{4,2},13}}$ be the subgroup consisting of matrices of the form $\begin{psmallmatrix}b^2 & * \\ 0 & a^4\end{psmallmatrix}$ and $\begin{psmallmatrix} 2 & 0 \\ 0 & 4 \end{psmallmatrix}$,
\item[$\bullet$] let $\customlabel{H1351}{H_{\brk{5,1},13}}$ be the subgroup consisting of matrices of the form $\begin{psmallmatrix}a^4 & * \\ 0 & *\end{psmallmatrix}$,
\item[$\bullet$] let $\customlabel{H1352}{H_{\brk{5,2},13}}$be the subgroup consisting of matrices of the form $\begin{psmallmatrix}a^4 & * \\ 0 & b^2\end{psmallmatrix}$ and $\begin{psmallmatrix} 4 & 0 \\ 0 & 2 \end{psmallmatrix}$.
\end{enumerate}
Each of the groups $G_{\brk{i,13}}$ contain $-I$, and the groups $H_{\brk{i,j},13}$ do not contain $-I$. Moreover, we have $G_{i,13} = \pm H_{\brk{i,j},13}$. 
\vspace*{-1cm}   
\begin{figure}[h]
$$
\begin{tikzcd}[back line/.style={densely dotted},column sep = 2em, row sep = 3em]
{} & {} & {} & \brk{I} & {} & {} \\
{} & \ref{H1351}  \arrow[ dash]{urr}{468}& \ref{H1352}  \arrow[ dash, swap]{ur}{468}& {} & \ref{H1341}  \arrow[dash]{ul}{468}& \ref{H1342} \arrow[ dash, swap]{ull}{468}\\ 
{}& \ref{G133} \arrow[dash, bend left = 125, near start]{uurr}{624}&  \ref{G135}\arrow[dash, near end, swap]{ul}{2}\arrow[dash, near end]{u}{2}& \ref{G132} \arrow[dash]{uu}{624}& \ref{G134} \arrow[dash, near end, swap]{ur}{2}\arrow[dash, near end]{u}{2}& \ref{G131} \arrow[ dash, bend right  = 120, swap]{uull}{624}  \\
\widetilde{\ref{G137}} \arrow[dash, bend left = 100, near start]{rrruuu}{288}& {}& {} & \ref{G136} \arrow[swap,dash]{rru}{3} \arrow[dash]{u}{3} \arrow[dash]{llu}{3} \arrow[dash]{ru}{2} \arrow[dash, swap]{lu}{2}& {} & {} \\
{} & {} & {} & \GL_2(\ZZ{13})  \arrow[dash]{lllu}{91} \arrow[dash]{u}{14} & {} & {}  
\end{tikzcd}
$$
\caption{Applicable subgroup lattice for $\GL_2(\ZZ{13})$}
\end{figure}
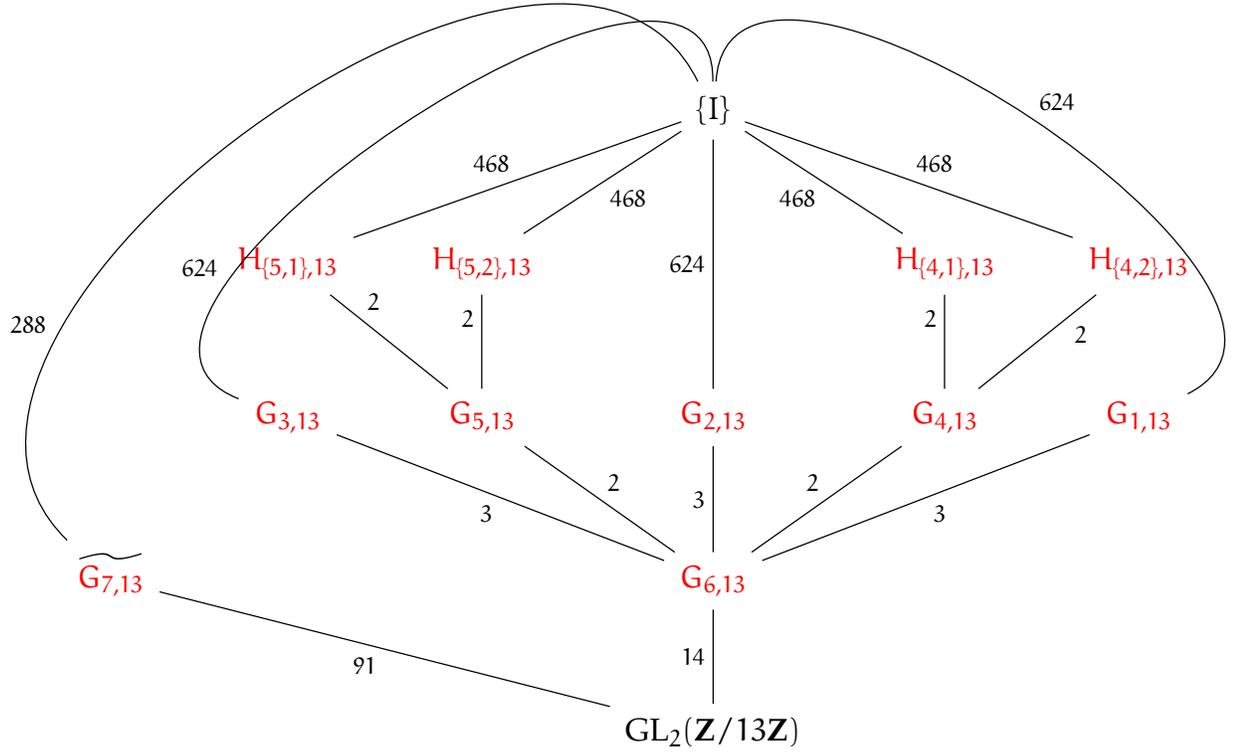

Define the polynomials
\begin{align*}
P_1(t) &= \begin{array}{c}
t^{12}+231t^{11}+269t^{10}-3160t^9+6022t^8-9616t^7+21880t^6 \\ -34102t^5+28297t^4-12455t^3+2876t^2-243t+1
\end{array}, \\
P_2(t) &= t^{12}-9t^{11}+29t^{10}-40t^9+22t^8-16t^7+40t^6-22t^5-23t^4+25t^3-4t^2-3t+1, \\
P_3(t) &= 	(t^4-t^3+2t^2-9t+3)(3t^4-3t^3-7t^2+12t-4)(4t^4-4t^3-5t^2+3t-1), \\ 
P_4(t) &= 	t^8+235t^7+1207t^6+955t^5+3840t^4-955t^3+1207t^2-235t+1, \\
P_5(t) &= 	t^8-5t^7+7t^6-5t^5+5t^3+7t^2+5t+1, \\
P_6(t) &= 	t^4+7t^3+20t^2+19t+1.
\end{align*}
From \cite[Theorem~1.8(ii)]{zywinapossible}, $\rho_{E,13}(G_{\mQ})$ is conjugate in $\GL_2(\ZZ{13})$ to $G_{i,13}$ if and only if $j_E$ is of the form:
\begin{align*}
J_1(t) &= \frac{(t^2 - t + 1)^3P_1(t)^3}{(t-1)t(t^3 - 4t^2 + t + 1)^{13}} &J_2(t)&= \frac{(t^2 - t +1)^3P_2(t)^3}{(t-1)^{13}t^{13}(t^3 - 4t^2 + t +1)} \\
J_3(t) &= -\frac{13^4(t^2 - t + 1)^3P_3(t)^3}{((t^3-4t^2+t+1)^{13}(5t^3-7t^2-8t+5))} &J_4(t)&= \frac{(t^4 - t^3 + 5t^2 + t + 1)P_4(t)^3}{t(t^2 - 3t - 1)^{13}} \\
J_5(t) &= \frac{(t^4 - t^3 + 5t^2 + t + 1)P_5(t)^3}{t^{13}(t^2 - 3t - 1)} &J_6(t)& = \frac{(t^2 + 5t + 13)P_6(t)^3}{t}
\end{align*}
for some $t\in \mQ$ and each respective $i$. Furthermore, \cite[Theorem~1.8(iii)]{zywinapossible} gives explicit conditions on when $\rho_{E,13}(G_{\mQ})$ is conjugate to $H_{\brk{i,j},13}$ for $i = 4,5$ and $j = 1,2$, and \cite[Theorem~1.8(iv)]{zywinapossible} gives necessary numerical conditions for when $\rho_{E,13}(G_{\mQ})$ is conjugate to $G_{7,13}$. The case $\ell = 13$ is the first case for which Zywina does not give a complete description, which is due to three outstanding cases (see Section \ref{subsec:cursed}). Furthermore, the author gives equations for the modular curves $X_H(\ell)$ of level $\ell$ where $\ell$ is a primes $\leq 37$ and $H$ is an applicable subgroup of $\GL_2(\ZZ{\ell})$. 

\clearpage
\bibliographystyle{alpha}
\def\bibfont{\small}
\bibliography{CI.bib} 

\newcommand{\etalchar}[1]{$^{#1}$}
\def\polhk#1{\setbox0=\hbox{#1}{\ooalign{\hidewidth
  \lower1.5ex\hbox{`}\hidewidth\crcr\unhbox0}}}
\begin{thebibliography}{{LMF}13}

\bibitem[Ade01]{adelmann2001decomposition}
Clemens Adelmann.
\newblock {\em The decomposition of primes in torsion point fields}, volume
  1761.
\newblock Springer Science \& Business Media, 2001.

\bibitem[Bar14]{baran2014exceptional}
Burcu Baran.
\newblock An exceptional isomorphism between modular curves of level 13.
\newblock {\em Journal of Number Theory}, 145:273--300, 2014.

\bibitem[BC14]{banwait2014tetrahedral}
Barinder Banwait and John Cremona.
\newblock Tetrahedral elliptic curves and the local-global principle for
  isogenies.
\newblock {\em Algebra \& Number Theory}, 8(5):1201--1229, 2014.

\bibitem[BCP97]{MR1484478}
Wieb Bosma, John Cannon, and Catherine Playoust.
\newblock The {M}agma algebra system. {I}. {T}he user language.
\newblock {\em J. Symbolic Comput.}, 24(3-4):235--265, 1997.
\newblock Computational algebra and number theory (London, 1993).

\bibitem[BDM{\etalchar{+}}17]{balakrishnanetal_Splitcartan13}
Jennifer~S. Balakrishnan, Netan Dogra, J.~Steffen M{\"u}ller, Jan Tuitman, and
  Jan Vonk.
\newblock Explicit {C}habauty-{K}im for the {S}plit {C}artan {M}odular {C}urve
  of {L}evel {$13$}.
\newblock {\em Preprint}, (November 15, 2017).

\bibitem[BJ16]{braujones2014elliptic}
Julio Brau and Nathan Jones.
\newblock Elliptic curves with 2-torsion contained in the 3-torsion field.
\newblock {\em Proceedings of the American Mathematical Society},
  144(3):925--936, 2016.

\bibitem[Bru03]{bruin2003chabauty}
Nils Bruin.
\newblock Chabauty methods using elliptic curves.
\newblock {\em J. Reine Angew. Math.}, 562:27--49, 2003.

\bibitem[Bru08]{bruin2008arithmetic}
Nils Bruin.
\newblock The arithmetic of {P}rym varieties in genus 3.
\newblock {\em Compositio Mathematica}, 144(02):317--338, 2008.

\bibitem[BS10]{bruin2010MWsieve}
Nils Bruin and Michael Stoll.
\newblock The {M}ordell--{W}eil sieve: proving non-existence of rational points
  on curves.
\newblock {\em LMS Journal of Computation and Mathematics}, 13:272--306, 2010.

\bibitem[CG89]{coombes1989heterogeneous}
Kevin~R. Coombes and David~R. Grant.
\newblock On heterogeneous spaces.
\newblock {\em Journal of the London Mathematical Society}, 2(3):385--397,
  1989.

\bibitem[Cha41]{chabauty1941points}
Claude Chabauty.
\newblock Sur les points rationnels des courbes alg{\'e}briques de genre
  sup{\'e}rieura {l'}unit{\'e}.
\newblock {\em CR Acad. Sci. Paris}, 212:882--885, 1941.

\bibitem[CL07]{cremona2007finding}
John Cremona and Mark Lingham.
\newblock Finding all elliptic curves with good reduction outside a given set
  of primes.
\newblock {\em Experimental Mathematics}, 16(3):303--312, 2007.

\bibitem[Col85]{coleman1985effective}
Robert~F. Coleman.
\newblock Effective {C}habauty.
\newblock {\em Duke Math. Journal}, 52(3):765--770, 1985.

\bibitem[Duk97]{duke1997elliptic}
William Duke.
\newblock Elliptic curves with no exceptional primes.
\newblock {\em Comptes rendus de l'Acad{\'e}mie des sciences. S{\'e}rie 1,
  Math{\'e}matique}, 325(8):813--818, 1997.

\bibitem[FW01]{flynn2001covering}
E.~Victor Flynn and Joseph~L. Wetherell.
\newblock Covering collections and a challenge problem of {S}erre.
\newblock {\em Acta Arithmetica}, 98(2):197--205, 2001.

\bibitem[Gau66]{gauss1966disquisitiones}
Carl~Friedrich Gauss.
\newblock {\em Disquisitiones arithmeticae}, volume 157.
\newblock Yale University Press, 1966.

\bibitem[Gre10]{greicius2010}
Aaron Greicius.
\newblock Elliptic curves with surjective ad\'elic {G}alois representations.
\newblock {\em Experimental Mathematics}, 19(4):495--507, 2010.

\bibitem[Hal98]{halberstadt1998courbe}
Emmanuel Halberstadt.
\newblock Sur la courbe modulaire {$X_{\text{nd\'ep}}(11)$}.
\newblock {\em Experimental Mathematics}, 7(2):163--174, 1998.

\bibitem[HS00]{hindry2000diophantine}
Marc Hindry and Joseph~H. Silverman.
\newblock {\em Diophantine geometry: an introduction}, volume 201.
\newblock Springer Science \& Business Media, 2000.

\bibitem[Jon10]{jones2010almost}
Nathan Jones.
\newblock Almost all elliptic curves are {S}erre curves.
\newblock {\em Transactions of the American Mathematical Society},
  362(3):1547--1570, 2010.

\bibitem[Kat81]{katz_Galoisproperties}
Nicholas~M. Katz.
\newblock Galois properties of torsion points on abelian varieties.
\newblock {\em Invent. Math.}, 62(3):481--502, 1981.

\bibitem[KRZB16]{katz2016diophantine}
Eric Katz, Joseph Rabinoff, and David Zureick-Brown.
\newblock Diophantine and tropical geometry, and uniformity of rational points
  on curves.
\newblock {\em arXiv:1606.09618}, 2016.

\bibitem[Lan87]{lang1987elliptic}
Serge Lang.
\newblock {\em Elliptic functions}.
\newblock Springer, 1987.

\bibitem[{LMF}13]{lmfdb}
The {LMFDB Collaboration}.
\newblock The l-functions and modular forms database.
\newblock \url{http://www.lmfdb.org}, 2013.
\newblock [Online; accessed 16 September 2013].

\bibitem[Maz77]{mazur1977rational}
Barry Mazur.
\newblock Rational points on modular curves.
\newblock In {\em Modular functions of one variable V}, pages 107--148.
  Springer, 1977.

\bibitem[Mor17]{JSM2017}
Jackson~S. Morrow.
\newblock Electronic transcript of computations for the manuscript
  ``{C}omposite images of {G}alois for elliptic curves over {$\mathbf{Q}$} \&
  {E}ntanglement fields", 2017.
\newblock {A}vailable at
  \texttt{\url{https://github.com/jmorrow4692/CompositeLevelandEntanglemets}}.

\bibitem[MP07]{mccallum2007method}
William McCallum and Bjorn Poonen.
\newblock The method of {C}habauty and {C}oleman.
\newblock {\em preprint}, 11, 2007.

\bibitem[Rib76]{ribet1976galois}
Kenneth~A. Ribet.
\newblock Galois action on division points of abelian varieties with real
  multiplications.
\newblock {\em American Journal of mathematics}, pages 751--804, 1976.

\bibitem[RS01]{rubin2001mod}
Karl Rubin and Alice Silverberg.
\newblock Mod 2 representations of elliptic curves.
\newblock {\em Proceedings of the American Mathematical Society}, pages 53--57,
  2001.

\bibitem[RZB]{RZB}
Jeremy Rouse and David Zureick-Brown.
\newblock Electronic transcript of computations for the paper {``E}lliptic
  curves over $\mathbf{Q}$ and 2-adic images of {G}alois{"}.
\newblock {A}vailable at \texttt{\url{http://users.wfu.edu/rouseja/2adic/}}.

\bibitem[RZB15]{rouse2014elliptic}
Jeremy Rouse and David Zureick-Brown.
\newblock Elliptic curves over {$\mathbf{Q}$} and 2-adic images of {G}alois.
\newblock {\em Res. Number Theory}, 1, 2015.

\bibitem[Ser72]{serre1972}
Jean-Pierre Serre.
\newblock Propri\'et\'es galoisiennes des points d'ordre fini des courbes
  elliptiques.
\newblock {\em Inventiones mathematicae}, 15(4):259 331, 1972.

\bibitem[Sik09]{siksek2009chabauty}
Samir Siksek.
\newblock Chabauty for symmetric powers of curves.
\newblock {\em Algebra \& Number Theory}, 3(2):209--236, 2009.

\bibitem[Sil09]{silvermanAEC}
Joseph~H. Silverman.
\newblock {\em The arithmetic of elliptic curves}, volume 106.
\newblock Springer, 2009.

\bibitem[Sko01]{skorobogatov2001torsors}
Alexei Skorobogatov.
\newblock {\em Torsors and rational points}.
\newblock Number 144. Cambridge University Press, 2001.

\bibitem[Sto06]{stoll2006independence}
Michael Stoll.
\newblock Independence of rational points on twists of a given curve.
\newblock {\em Compositio Mathematica}, 142(05):1201--1214, 2006.

\bibitem[SZ17]{sutherland5modular}
Andrew~V. Sutherland and David Zywina.
\newblock Modular curves of genus zero and prime-power level.
\newblock {\em Algebra \& Number Theory}, to appear, 2017.

\bibitem[Wet97]{wetherell1997bounding}
Joseph~Loebach Wetherell.
\newblock {\em Bounding the number of rational points on certain curves of high
  rank}.
\newblock PhD thesis, University of California, Berkeley, 1997.

\bibitem[Zyw10a]{zywina10maximal}
David Zywina.
\newblock Elliptic curves with maximal {G}alois action on their torsion points.
\newblock {\em Bulletin of the London Mathematical Society}, 42(5):811--826,
  2010.

\bibitem[Zyw10b]{zywina2010hilbert}
David Zywina.
\newblock Hilbert's irreducibility theorem and the larger sieve.
\newblock {\em arXiv preprint arXiv:1011.6465}, 2010.

\bibitem[Zyw11]{zywina2011surjectivity}
David Zywina.
\newblock {O}n the surjectivity of mod $\ell$ representations associated to
  elliptic curves.
\newblock {\em (preprint)}, 2011.

\bibitem[Zyw15]{zywinapossible}
David Zywina.
\newblock On the possible images of the mod $\ell$ representations associated
  to elliptic curves over $\mathbf{Q}$.
\newblock 2015.
\newblock {A}vailable at
  \texttt{\url{http://www.math.cornell.edu/~zywina/papers/PossibleImages/index.html}}.

\end{thebibliography}
\end{document}